\documentclass[psamsfonts]{amsart}
\usepackage{amssymb}
\usepackage{graphicx,color}
\usepackage{amssymb,amsfonts}
\usepackage[all,arc]{xy}
\usepackage{enumerate}
\usepackage{mathrsfs}

\newtheorem{thm}{Theorem}[section]
\newtheorem{lemma}[thm]{Lemma}
\newtheorem{prop}[thm]{Proposition}
\newtheorem{cor}[thm]{Corollary}

\theoremstyle{definition}

\theoremstyle{remark}
\newtheorem{rem}[thm]{Remark}
\makeatletter
\let\c@equation\c@thm
\makeatother
\numberwithin{equation}{section}

\title[Hypersurfaces of Prescribed Curvature and Boundary ]{Strictly Locally Convex Hypersurfaces with Prescribed Curvature and Boundary in Space Forms}

\author{Zhenan Sui}

\address[Zhenan Sui]{Institute for Advanced Study in Mathematics of HIT, Harbin Institute of Technology}

\begin{document}

\begin{abstract}
This paper is devoted to $C^2$ a priori estimates for strictly locally convex radial graphs with prescribed Weingarten curvature and boundary in space forms. By constructing two-step continuity process and applying degree theory arguments,  existence results in space forms are established for prescribed Gauss curvature equation under the assumption of a strictly locally convex subsolution.
\end{abstract}

\maketitle

%\tableofcontents

\section {\large Introduction}

\vspace{4mm}

In $(n+1)$-dimensional space form $N^{n+1}(K)$ with $n \geq 2$, given a disjoint collection $\Gamma = \{ \Gamma_1, \ldots, \Gamma_m \}$
of smooth closed embedded $(n - 1)$-dimensional submanifolds, and a smooth positive function $\psi$, we are concerned with the general Plateau type problem for strictly locally convex hypersurfaces $\Sigma$ determined by the curvature equation
\begin{equation}  \label{eq1-0}
\sigma_k (\kappa [ \Sigma ]) = \psi(V, \,\nu)
\end{equation}
as well as the boundary condition
\begin{equation} \label{eq1-2}
\partial \Sigma = \Gamma
\end{equation}
where $\kappa [\Sigma] = (\kappa_1, \ldots, \kappa_n)$ denotes the principal curvatures of $\Sigma$, $V$ is a conformal Killing field which will be specified later, $\nu$ is the unit outer normal field to $\Sigma$, and
\[\sigma_k (\lambda) = \sum\limits_{1 \leq i_1 < \ldots < i_k \leq n} \,\lambda_{i_1} \cdots \lambda_{i_k}\]
is the $k$-th elementary symmetric function defined on $k$-th G\r arding's cone
\[ \Gamma_k = \{ (\lambda_1, \ldots, \lambda_n) \in \mathbb{R}^n \,\vert \, \sigma_j (\lambda) > 0,\,\, j = 1, \ldots, k \} \]
$\sigma_k (\kappa [\Sigma])$ is called the $k$-th Weingarten curvature of $\Sigma$. In particular, the $1$st, $2$nd and $n$-th Weingarten curvature are the well known mean curvature, scalar curvature and Gauss curvature respectively. We say $\Sigma$ is strictly locally convex if its principal curvatures are all positive everywhere in $\Sigma$, and $\Sigma$ is $k$-admissible if $\kappa [ \Sigma ] \in \Gamma_k$.

The space form $N^{n+1}(K)$ with constant sectional curvature $K = 0$, $1$ or $- 1$ can be modeled as follows (see for example \cite{SX15}). In $\mathbb{R}^{n + 1}$, fix the origin $0$ and let $\mathbb{S}^n$ denote the unit sphere centered at $0$. Choose the spherical coordinates $(z, \rho)$ in $\mathbb{R}^{n + 1}$ with $z \in \mathbb{S}^n$. Define
\[ \bar{g} : = d \rho^2 + \phi^2(\rho)\, \sigma \]
where $\sigma$ is the standard metric on $\mathbb{S}^n$ induced from $\mathbb{R}^{n + 1}$ and
\begin{equation*}
\phi(\rho) = \left\{\begin{aligned} \rho, \quad\,\, \mbox{on} \,\,\, [ \,0, \infty) \\
\sin(\rho), \quad\,\, \mbox{on} \,\,\, [ \,0, \frac{\pi}{2}) \\
\sinh(\rho), \quad\,\, \mbox{on} \,\,\, [ \,0, \infty) \end{aligned}\right.
\end{equation*}
Then $(\mathbb{R}^{n + 1}, \bar{g})$ is a model of $N^{n + 1}(K)$, which is $\mathbb{R}^{n + 1}$, $\mathbb{S}_{+}^{n + 1}$ or $\mathbb{H}^{n + 1}$ depending on $K = 0$, $1$ or $ - 1$. Let
$V = \phi(\rho) \,\frac{\partial}{\partial \rho}$.
It is well know that $V$  is a conformal Killing field in $ N^{n+1}(K)$. In Euclidean space $\mathbb{R}^{n + 1}$, it is just the position vector field.

For starshaped compact hypersurfaces, the early influential work includes \cite{Oliker, CNSIV} for Euclidean space,  Jin-Li \cite{JL05} for hyperbolic space, \cite{BLO, LO02} for elliptic space. If $\psi$ is allowed to depend arbitrarily on $\nu$,
the most current breakthrough is due to Guan-Ren-Wang \cite{GRW15}, where the authors studied Weingarten curvature in Euclidean space (see also Spruck-Xiao \cite{SX15} for scalar curvature in space forms and Chen-Li-Wang \cite{CLW18} for Weingarten curvature in warped product spaces).

For Dirichlet problem in $\mathbb{R}^{n+1}$, Caffarelli-Nirenberg-Spruck \cite{CNSV} initiated  the study of vertical graphs over strictly convex domains in $\mathbb{R}^n$. Later, Guan-Spruck \cite{GS93} studied radial graphs in $\mathbb{R}^{n+1}$ of constant Gauss-Kronecker curvature, where they removed the convexity assumption of the domain, but instead proposed a subsolution condition. This subsolution assumption is later widely used when discussing Dirichlet boundary value problems for general curvature equations (as well as Hessian type equations), see for instance \cite{Guan95, Guan98, GS04, Su16, Cruz}.

A strictly locally convex hypersurface with boundary may not be convex globally; it locally lies on one side of its tangent plane at any point, which may be very complicated in general.
In this paper, we focus on those that can be represented as radial graphs over some domain in $\mathbb{S}^n$. Therefore we assume $\Gamma$ to be the boundary of a smooth positive radial graph $\varphi$ defined on a smooth domain $\Omega \subset \mathbb{S}^n$, i.e., $\Gamma = \{ (z, \varphi(z))\, | z \in \partial \Omega \}$.  We seek a smooth strictly locally convex radial graph $\Sigma = \{ (z, \rho(z))\, | z \in \Omega \}$ satisfying the Dirichlet problem
\begin{equation} \label{eq1-1}
\left\{\begin{aligned}
\sigma_k(\kappa [ \rho ])\, = \,\,&  \psi(z, \rho, \nabla' \rho) \quad & \mbox{in} \quad \Omega
\\
\rho \, = \,\, & \varphi \quad \quad & \mbox{on} \quad \partial \Omega
\end{aligned}\right.
\end{equation}
where $\nabla'$ is the Levi-Civita connection on $\mathbb{S}^n$ with respect to $\sigma$ and we use the same $\psi$ for the right hand side.

The first main result in this paper is the following $C^2$ estimate, which
is a crucial step for proving existence and higher order regularity of solutions.

\begin{thm} \label{Theorem1-1}
In space form $N^{n+1}(K)$, suppose that
\begin{equation} \label{eq1-9}
\Omega \,\,\, \mbox{does}\,\,\mbox{not}\,\,\mbox{contain}\,\,\mbox{any}\,\,\mbox{hemisphere}
\end{equation}
and $\Gamma$ can span a positive radial graph $\overline{\rho} \in C^2 (\overline{\Omega})$ in $N^{n + 1}(K)$  which is strictly locally convex in a neighborhood of $\Gamma$. Then for any strictly locally convex  radial graph $\rho \in C^4(\Omega) \cap C^2(\overline{\Omega})$ satisfying \eqref{eq1-1}  with $\rho \leq \overline{\rho}$ in $\Omega$, we have
\[ \Vert \rho \Vert_{C^2(\overline{\Omega})} \,\leq \, C    \]
where $C$ depends only on $n$, $k$, $\Omega$, $\Vert\psi\Vert_{C^2}$, $\Vert \overline{\rho}\Vert_{C^1(\overline{\Omega})}$,  $\Vert \varphi \Vert_{C^4(\overline{\Omega})}$, $\inf \psi$, $\inf_{\partial\Omega}\overline{\rho}$ and the convexity of $\overline{\rho}$.
\end{thm}

When $k = n$, $C^2$ estimates have been derived in $\mathbb{R}^{n+1}$ by \cite{Su16, GS93, Guan95} and in $\mathbb{S}^{n+1}_+$ by Lim \cite{Lim}, while for all $k$ when $\psi$ does not depend on $\nu$ in $\mathbb{R}^{n+1}$  by Cruz \cite{Cruz}. In \cite{GRW15}, Guan-Ren-Wang solved the longstanding problem on global $C^2$ estimates for convex hypersurfaces in $\mathbb{R}^{n+1}$ subject to equation \eqref{eq1-0}. They also removed the convexity assumption and instead considered starshaped $2$-admissible hypersurfaces in the case $k = 2$, the proof of which was later simplified by Spruck-Xiao \cite{SX15} in space forms. On the other hand, \cite{GRW15} provides counterexamples to indicate that global $C^2$ estimates for general $\psi$ depending on both $V$ and $\nu$ do not hold for curvature quotient equations. In this paper, we extend the estimates in \cite{GRW15}
to space forms. For $C^2$ boundary estimates, it is necessary in Theorem \ref{Theorem1-1} to assume $\overline{\rho}$ to be strictly locally convex near its boundary, for otherwise there are topological obstructions to the existence of strictly locally convex hypersurfaces spanning a given $\Gamma$ (see \cite{Ro93}); besides, the convexity assumption on the prescribed hypersurfaces can not be weakened, even for the case $k = 2$, or when $\psi$ does not depend on $\nu$ (see section 3).
The significance of Theorem \ref{Theorem1-1} lies in the arbitrary dependence of $\psi$ on $\nu$ for all $k$ as well as a unified approach in different space forms via change of variable for Plateau type problems.

To establish existence results, we confine ourselves to prescribed Gauss curvature equation, i.e. the case $k = n$, because for general Weingarten curvature equation, a positive lower bound of principal curvatures may not be obtained and the convexity may not be preserved during the continuity process. Applying Theorem \ref{Theorem1-1}, we can prove the following existence results.

\begin{thm} \label{Theorem1-2}
Under condition \eqref{eq1-9}, assume that there exists a positive strictly locally convex radial graph  $\overline{\rho} \in C^2 (\overline{\Omega})$ in $N^{n+1}(K)$ satisfying
\begin{equation} \label{eq1-4}
\left\{ \begin{aligned}
\sigma_n(\kappa [ \overline{\rho} ]) & \, \geq \, \psi( z, \overline{\rho}, \nabla' \overline{\rho}) \quad & \mbox{in} \quad \Omega
\\
\overline{\rho} & \, =  \,  \varphi \quad\quad & \mbox{on} \quad \partial \Omega
\end{aligned} \right.
\end{equation}
Then there exists a smooth strictly locally convex radial graph $\Sigma = \{ (z, \rho(z))\,\vert\, z \in \Omega \} \subset N^{n+1}(K)$ satisfying the Dirichlet problem \eqref{eq1-1} for $k = n$ with $\rho \leq \overline{\rho}$ in $\overline{\Omega}$ and uniformly bounded principal curvatures
\[ 0 < K_0^{-1} \leq \kappa_i \leq K_0 \quad \mbox{on} \quad \Sigma, \]
where $K_0$ is a uniform positive constant depending only on $n$, $\Omega$, $\Vert\psi\Vert_{C^2}$, $\Vert \overline{\rho}\Vert_{C^1(\overline{\Omega})}$,  $\Vert \varphi \Vert_{C^4(\overline{\Omega})}$, $\inf \psi$, $\inf_{\partial\Omega}\overline{\rho}$ and the convexity of $\overline{\rho}$.
\end{thm}

The existence results in $\mathbb{R}^{n+1}$ are proved in \cite{Su16, GS93, Guan95}. The main issue in proving existence for radial graphs is due to the nontrivial kernel of the linearized operator, since continuity method can not be applied directly. In \cite{GS93} for prescribed constant Gauss curvature, the authors used monotone iteration approach to overcome this difficulty and hence uniform $C^2$ estimates for the monotone sequence are conducted. Based on this result, in \cite{Guan95}, more $C^2$ estimates are derived for a wider class of auxiliary equations in order to obtain existence results for general $\psi$. In contrast, Su \cite{Su16} provided a more efficient way by observing that there exist auxiliary equations in $\mathbb{R}^{n+1}$ with invertible linearized operators which can be found after a change of variable. The author then constructed a two-step continuity process for applying continuity method and degree theory.

In this paper, we generalize Su's idea to $\mathbb{H}^{n+1}$.
In $\mathbb{S}^{n+1}_+$, however, there is no auxiliary equation with invertible linearized operator, neither can we apply monotone iteration approach since global $C^2$ estimates for the monotone sequence are too cumbersome to derive. In this paper, we create a new continuity process starting from an auxiliary equation in $\mathbb{R}^{n+1}$ with invertible linearized operator and ending up with our concerned equation in $\mathbb{S}^{n+1}_+$.
In more curved ambient spaces, this paper is not the first to study the Dirichlet problem for prescribed curvature equations. In fact, \cite{Neh98, Neh99} initiated the study in general Riemannian manifolds, where the author assumed the subsolution to be part of a special closed convex hypersurface.
Theorem \ref{Theorem1-2} is brand new in space forms in the sense that, the assumption of a strictly locally convex subsolution is much weaker than \cite{Neh98, Neh99}, and our proof is completely different and global, which is based on the geometry of space forms and the degree theory developed in \cite{Li89}.

This paper is organized as follows: in section 2, we reformulate equation \eqref{eq1-1} by change of variable in two different ways: \eqref{eq3-14} is designed for deriving $C^2$ boundary estimates in section 3 while \eqref{eq6-2} is for proving existence in section 5 and 6. Section 4 is devoted to global curvature estimates.

The author would like to thank Dr. Wei Sun for enlightening discussions. The author also thanks the reviewer for careful reading and insightful comments and suggestions, which resulted in improvement of Theorem \ref{Theorem1-1}. This work is supported by the grant (no. AUGA5710000618) from Harbin Institute of Technology.

\vspace{5mm}

\section{Strictly locally convex radial graphs in space forms}

\vspace{3mm}

Throughout this paper, we focus on hypersurface $\Sigma \subset N^{n + 1} (K)$ that can be represented as a radial graph over a smooth domain $\Omega \subset \mathbb{S}^n$, i.e. $\Sigma$ can be expressed as
\[ \Sigma  = \{ (z, \rho(z))\,| z \in \Omega \subset \mathbb{S}^n \} \]
The range for $\rho = \rho(z)$ is $ (0, \rho_U^K)$ where
\begin{equation} \label{eq3-15}
\rho_U^K = \left\{ \begin{aligned} & \infty,\quad\quad\mbox{if} \quad K = 0 \quad \mbox{or}\quad - 1 \\
& \frac{\pi}{2},\quad\quad\mbox{if} \quad K = 1
\end{aligned}
\right.
\end{equation}

First recall the related geometric objects on $\Sigma$. Let $e_1, \ldots, e_n$ be a local orthonormal frame on $\mathbb{S}^n$,
following the notations in \cite{SX15}, the induced metric, its inverse, unit normal, and second fundamental form on $\Sigma$ are given respectively by
\begin{equation} \label{eq3-1}
 g_{ij} = \phi^2\, \delta_{ij} + \rho_i \rho_j
\end{equation}
\begin{equation} \label{eq3-2}
 g^{ij} = \frac{1}{\phi^2} \big( \delta_{ij} - \frac{\rho_i \rho_j}{\phi^2 + |\nabla'\rho|^2} \big)
\end{equation}
\begin{equation} \label{eq3-4}
 \nu = \frac{- \nabla' \rho + \phi^2 \,\frac{\partial}{\partial \rho}}{\sqrt{\phi^4 + \phi^2 |\nabla' \rho|^2}}
\end{equation}
\begin{equation} \label{eq3-5}
 h_{ij} = \frac{\phi}{\sqrt{\phi^2 + |\nabla' \rho|^2}} \,\big( - \nabla'_{ij} \rho + \frac{2 \phi'}{\phi}\,\rho_i \rho_j + \phi \phi' \delta_{ij} \big)
\end{equation}
where $\rho_i = \rho_{e_i} = \nabla'_{e_i} \rho = \nabla'_{i} \rho $, $\rho_{ij} = \nabla'_{e_j} \nabla'_{e_i} \rho =  \nabla'_{e_j e_i} \rho = \nabla'_{j\,i} \rho$, etc. All other covariant derivatives are interpreted in this manner.
Thus $\nabla'\rho = \rho_k \, e_k$.

The principal curvatures $\kappa_1, \ldots, \kappa_n$ of the radial graph $\rho$ are the eigenvalues of the symmetric matrix $ \{ a_{ij} \}$:
\[ a_{ij} = \gamma^{i k} \,h_{k l}\, \gamma^{l j} \]
where $\{ \gamma^{ik} \}$ and its inverse $\{ \gamma_{ik} \}$ are given respectively by
\begin{equation} \label{eq3-6}
 \gamma^{ik} = \frac{1}{\phi} ( \delta_{ik} - \frac{\rho_i \,\rho_k}{\sqrt{\phi^2 + |\nabla' \rho|^2} ( \phi + \sqrt{\phi^2 + |\nabla' \rho|^2} )} )
\end{equation}
\begin{equation} \label{eq3-7}
\gamma_{ik} = \phi \,\delta_{ik} + \frac{\rho_i \rho_k}{ \phi + \sqrt{\phi^2 + |\nabla' \rho|^2} }
\end{equation}
In fact, $\{ \gamma_{ik} \}$ is the square root of the metric, i.e., $\gamma_{ik} \gamma_{kj} = g_{ij}$.

Obviously, $\Sigma$ is strictly locally convex
if and only if the symmetric matrix $\{ a_{ij }\}$ or $\{ h_{ij} \}$ is positive definite everywhere in $\Omega$.
For simplicity, we say a $C^2$ function $\rho$ is strictly locally convex if the hypersurface $\Sigma$ represented by $\rho$ is strictly locally convex.
Also, $a_{ij} > 0 $ (or $\geq 0$ ) means that the symmetric matrix $\{ a_{ij} \}$ is positive definite (or positive semi-definite); and $a_{ij} \geq b_{ij}$ means that the symmetric matrices $\{a_{ij}\}$ and $\{b_{ij}\}$ satisfy $ a_{ij} - b_{ij} \geq 0 $.

\vspace{5mm}

\subsection{Transformation for deriving a priori estimates}~

\vspace{3mm}

Now we change $\rho$ into $u$ for deriving $C^2$ boundary estimates in section 3.
Set
\begin{equation} \label{eq3-14}
\rho = \zeta ( u ) \, = \,\left\{ \begin{aligned} & \frac{1}{u},\quad\quad & \mbox{if} \quad  K = 0  \\
& \mbox{arccot}\, u,\quad\quad & \mbox{if} \quad  K = 1  \\
& \frac{1}{2} \ln \big( \frac{ u + 1 }{u - 1} \big) ,\quad\quad & \mbox{if} \quad  K = - 1
\end{aligned}
\right.
\end{equation}
According to \eqref{eq3-15}, the range for $u$ is  $(u_L^K, \infty)$ where
\begin{equation} \label{eq3-16}
u_L^K = \left\{ \begin{aligned} & 0,\quad\quad\mbox{if} \quad K = 0 \quad \mbox{or}\quad 1 \\
& 1,\quad\quad\mbox{if} \quad K = -1
\end{aligned}
\right.
\end{equation}
The formulas \eqref{eq3-1}, \eqref{eq3-2}, \eqref{eq3-6}, \eqref{eq3-7} and \eqref{eq3-5} can be expressed in terms of $u$,
\begin{equation} \label{eq3-8}
g_{ij} = \phi^2 \,\delta_{ij} + \zeta'^2(u)\, u_i u_j
\end{equation}

\begin{equation} \label{eq3-9}
g^{ij} = \frac{1}{\phi^2} \Big( \delta_{ij} - \frac{\zeta'^2(u) u_i u_j}{\phi^2 + \zeta'^2(u) |\nabla' u|^2} \Big)
\end{equation}

\begin{equation} \label{eq3-10}
\begin{aligned} %keep all
\gamma^{ik} =  & \,\frac{1}{\phi}\, \Big( \delta_{ik} - \frac{ \zeta'^2(u) u_i u_k }{\sqrt{\phi^2 + \zeta'^2(u) |\nabla' u|^2 } ( \phi + \sqrt{\phi^2 + \zeta'^2(u) |\nabla' u|^2 } )}\Big) \\  = & \left\{ \begin{aligned} & u \,\big( \delta_{ik} - \frac{ u_i u_k }{\sqrt{u^2 +  |\nabla' u|^2 } ( u + \sqrt{ u^2 +  |\nabla' u|^2 } )} \big)  ,&\quad\mbox{if} \,\, K = 0  \\
& \sqrt{ 1 + u^2 } \,\big( \delta_{ik} - \frac{ u_i u_k }{\sqrt{1 + u^2 +  |\nabla' u|^2 } ( \sqrt{ 1 +  u^2 } + \sqrt{ 1 + u^2 +  |\nabla' u|^2 } )} \big)  ,&\quad\mbox{if} \,\, K = 1 \\
& \sqrt{ u^2 - 1 } \,\big( \delta_{ik} - \frac{ u_i u_k }{\sqrt{ u^2 - 1 +  |\nabla' u|^2 } ( \sqrt{ u^2 - 1 } + \sqrt{ u^2 - 1 +  |\nabla' u|^2 } )} \big) , &\quad\mbox{if} \,\, K = - 1
\end{aligned}
\right.
\end{aligned}
\end{equation}

\begin{equation} \label{eq3-11}
 \gamma_{ik} =  \phi \,\delta_{ik} + \frac{\zeta'^2(u) u_i u_k}{ \phi + \sqrt{\phi^2 + \zeta'^2(u) |\nabla' u|^2 } }
\end{equation}

\begin{equation} \label{eq3-13} %keep all
\begin{aligned}
h_{ij} = \,& \frac{- \zeta'(u) \phi}{\sqrt{\phi^2 + \zeta'^2 |\nabla' u|^2}} ( \nabla'_{ij} u + u \,\delta_{ij} ) \\
       = \,& \left\{ \begin{aligned} & \frac{1}{ \sqrt{u^4 + u^2 |\nabla' u|^2}} (\nabla'_{ij} u + u \,\delta_{ij}),&\quad\mbox{if} \quad K = 0  \\
& \frac{1}{\sqrt{( 1 + u^2)^2 +  ( 1 + u^2 )  |\nabla' u|^2} } (\nabla'_{ij} u + u \,\delta_{ij}),&\quad\mbox{if} \quad K = 1 \\
& \frac{1}{\sqrt{ ( u^2 - 1 )^2 + ( u^2 - 1 ) |\nabla' u|^2} } (\nabla'_{ij} u + u \,\delta_{ij}), &\quad\mbox{if} \quad K = - 1
\end{aligned}
\right.
\end{aligned}
\end{equation}
Hence
\begin{equation} \label{eq3-18}
a_{ij} = \frac{- \zeta'(u) \phi}{\sqrt{\phi^2 + \zeta'^2 |\nabla' u|^2}} \gamma^{ik}\, ( \nabla'_{kl} u + u \,\delta_{kl} )\,\gamma^{lj}
\end{equation}
It is easy to see that  $\Sigma$ (or $u$) is strictly locally convex if and only if
\begin{equation} \label{eq3-17}
 \nabla'_{ij} u + \, u \,\delta_{ij}  > 0 \quad\mbox{in} \quad \Omega
\end{equation}

\vspace{5mm}

\subsection{ Transformation for proving existence }~

\vspace{3mm}

We further change $u$ into $v$ for proving existence in section 5 and 6.
Set
\begin{equation} \label{eq6-2}
u = \eta ( v ) \, = \,\left\{ \begin{aligned} & e^v,\quad\quad & \mbox{if} \quad  K = 0  \\
& \sinh v,\quad\quad & \mbox{if} \quad  K = 1  \\
& \cosh v,\quad\quad & \mbox{if} \quad  K = - 1
\end{aligned}
\right.
\end{equation}
According to \eqref{eq3-16}, the range for $v$ is $(v_L^K, \infty)$ where
\begin{equation} \label{eq6-19}
v_L^K = \left\{ \begin{aligned} & - \infty,\quad\quad &\mbox{if} \quad K = 0  \\
& 0,\quad\quad &\mbox{if} \quad K = 1 \,\, \mbox{or}\,\, -1
\end{aligned}
\right.
\end{equation}
The formula \eqref{eq3-10} and \eqref{eq3-13} become
\begin{equation} \label{eq6-3}
\gamma^{ik} =  \,\eta'(v) \,\Big( \delta_{ik} - \frac{ v_i v_k }{\sqrt{1 +  |\nabla' v|^2 } ( 1 + \sqrt{ 1 +  |\nabla' v|^2 } )} \Big)
\end{equation}
\begin{equation} \label{eq6-4}
h_{ij} = \, \frac{1}{ \eta'^2(v) \sqrt{1 + |\nabla' v|^2}} \Big( \eta'(v) \nabla'_{ij} v + \eta(v) v_i v_j +  \eta(v) \,\delta_{ij} \Big)
\end{equation}
Denoting
\[ w = \sqrt{1 + |\nabla' v|^2}\quad \mbox{and} \quad \tilde{\gamma}^{ik} = \delta_{ik} - \frac{ v_i v_k }{w ( 1 + w )},\]
we have
\begin{equation} \label{eq6-1}
\begin{aligned}
a_{ij} =  & \, \frac{1}{ w } \, \tilde{\gamma}^{ik} \Big( \eta'(v) \nabla'_{kl} v + \eta(v) v_k v_l +  \eta(v) \,\delta_{kl}\Big)\, \tilde{\gamma}^{lj} \\
= & \, \frac{1}{w} \Big(\, \eta(v)\, \delta_{ij} \,+ \,\eta'(v)\, \tilde{\gamma}^{ik}\, \nabla'_{kl} v  \,\tilde{\gamma}^{lj} \,\Big)
\end{aligned}
\end{equation}

\vspace{5mm}

\subsection{Reformulation of equation \eqref{eq1-1} under transformation \eqref{eq3-14}}~

\vspace{3mm}

Let $f = \sigma_k^{1/k}$. Note that $f$ satisfies the following properties (see \cite{CNSIII, CNSV}):
\begin{equation} \label{eq1-5}
\frac{\partial f}{\partial \lambda_i} > 0 \quad \mbox{in} \quad \Gamma_k,\quad i = 1, \ldots, n
\end{equation}
\begin{equation} \label{eq1-6}
f \,\,\mbox{is} \,\,\mbox{concave}\,\, \mbox{in} \,\, \Gamma_k
\end{equation}
\begin{equation} \label{eq1-7}
f > 0 \quad\,\, \mbox{in} \,\,\, \Gamma_k \quad\quad \mbox{and} \quad\quad  f = 0 \quad\,\,\mbox{on} \,\,\, \partial\Gamma_k
\end{equation}

Under transformation \eqref{eq3-14}, the Dirichlet problem \eqref{eq1-1}
is equivalent to
\begin{equation} \label{eq2-11}
\left\{ \begin{aligned}
f(\kappa[ u ]) \, = \, \,&  \psi(z, u, \nabla' u)  \quad  & \mbox{in} \quad \Omega
\\
u \, =  \, \,&\varphi  \quad \quad & \mbox{on} \quad \partial \Omega
\end{aligned} \right.
\end{equation}
where we use the same $\psi$ for the function on the right hand side, and $\varphi$ for the boundary value.
Denote $A[u] = \{ a_{ij} \}$ where $a_{ij}$ is given by \eqref{eq3-18},
$F( A ) = f(\lambda( A ))$ where $\lambda(A)$ denotes the eigenvalues of $A$,
and
\[ G(r, p, u) \,= \, F ( A ( r, p, u ) )\]
where $A( r, p, u )$ is obtained from $A[u]$ with $( r, p, u )$ in place of $(\nabla'^2 u, \nabla' u, u)$. Therefore, $\kappa[u] = \lambda( A [ u ] )$ and
equation \eqref{eq2-11} is equivalent to
\begin{equation} \label{eq2-13}
\left\{ \begin{aligned}
G(\nabla'^2 u, \nabla' u, u) \,=  \,\,& \psi(z, u, \nabla' u)  \quad  &\mbox{in} \quad \Omega
\\
u \, =  \, \,&\varphi  \quad \quad & \mbox{on} \quad \partial \Omega
\end{aligned} \right.
\end{equation}
We  recall some properties of the function $F$ and $G$ (see \cite{GS04} for instance).
Denote
\[ F^{ij} ( A ) = \frac{\partial F}{\partial a_{ij}} ( A ), \quad F^{ij, kl} (A) = \frac{\partial^2 F}{\partial a_{ij} \partial a_{kl}} (A),\]
\[ G^{ij}(r, p, u) = \frac{\partial G}{\partial r_{ij}}(r, p, u), \quad G^i(r, p, u) = \, \frac{\partial G}{\partial p_i}(r, p, u),  \quad G_u(r, p, u) = \, \frac{\partial G}{\partial u}(r, p, u),\]
\[\psi_u(z, u, p) =  \frac{\partial \psi}{\partial u}(z, u, p), \quad \psi^i(z, u, p) = \,\frac{\partial \psi}{\partial p_i} (z, u, p)  \]
The matrix $\{ F^{ij} (A) \}$ is symmetric with eigenvalues $f_1, \ldots, f_n$;
by \eqref{eq1-5}, $F^{ij} (A) > 0$ whenever $\lambda(A) \in \Gamma_k$; by \eqref{eq1-6}, $F$ is a concave function of $A$, i.e., the symmetric matrix $F^{ij, kl} (A) \leq 0 $ whenever $\lambda(A) \in \Gamma_k$.
The function $G$ has similar properties. In fact, from \eqref{eq3-18} we have
\begin{equation} \label{eq2-14}
 G^{ij} = \frac{\partial G}{\partial u_{ij}}  = \frac{\partial F}{\partial a_{kl}} \frac{\partial a_{kl}}{\partial u_{ij}} =  \frac{ - \phi \zeta'(u)}{\sqrt{\phi^2 + \zeta'^2(u) |\nabla' u|^2}} F^{kl} \gamma^{ik} \gamma^{jl}
\end{equation}
Thus the symmetric matrix $G^{ij} > 0$ if and only if $F^{ij} > 0$, which in particular implies that equation \eqref{eq2-13} is elliptic for strictly locally convex solutions.
Also by \eqref{eq3-18} we can calculate
\[ \frac{\partial^2 G}{\partial u_{ij} \partial u_{kl}} = \frac{\partial a_{pq}}{\partial u_{ij}} \,\frac{\partial^2 F}{\partial a_{pq} \partial a_{rs}}\,\frac{\partial a_{rs}}{\partial u_{kl}}\]
which implies that $G$ is concave with respect to $\{ u_{ij} \}$ for strictly locally convex $u$.

We next compute $G^s$ and $G_u$.

\vspace{2mm}

\begin{lemma} \label{Lemma2-1}
Denote $w = \sqrt{ \phi^2  + \zeta'^2(u) |\nabla' u|^2}$. Then
\begin{equation*}
 G^s \, = \, - \frac{2 \zeta'^2 ( w \,\gamma^{is}\, u_q + \phi \,\gamma^{q s} u_i )}{ w ( \phi + w )} F^{ij} a_{q j} - \frac{\zeta'^2\, u_s}{w^2} \, F^{ij} a_{ij}
 \end{equation*}
\begin{equation*}
\begin{aligned}
G_u = -2 \Big( \phi \phi' \zeta' g^{iq} + \frac{\zeta' \zeta'' u_i u_q}{w^2}\Big) F^{ij} a_{q j} + \Big( \frac{\phi' \zeta'}{\phi} - \frac{\phi \phi' \zeta'}{w^2} + \frac{\phi^2 \zeta''}{\zeta'\, w^2} \Big) F^{ij} a_{ij} - \frac{\phi \zeta'}{w} F^{ij} g^{ij}
\end{aligned}
\end{equation*}
\end{lemma}

\begin{proof}
\begin{equation}  \label{eq2-16}
G^s \, = \frac{\partial F}{\partial a_{ij}} \frac{\partial a_{ij}}{\partial u_s} = F^{ij} \big( 2\, \frac{\partial \gamma^{ik}}{\partial u_s} \, h_{kl}\,\gamma^{lj} + \gamma^{ik}\,\frac{\partial h_{kl}}{\partial u_s}\,\gamma^{lj} \big)
\end{equation}
where
\begin{equation} \label{eq2-17}
\frac{ \partial \gamma^{ik}}{\partial u_s} = - \gamma^{ip}\, \frac{\partial \gamma_{pq}}{\partial u_s} \, \gamma^{qk}
\end{equation}
From \eqref{eq3-11} and \eqref{eq3-10},
\begin{equation} \label{eq2-18}
\frac{\partial \gamma_{pq}}{\partial u_s} \,=
   \frac{ \zeta'^2(u) ( \delta_{ps} u_q + \delta_{q s} u_p ) }{ \phi + w } - \frac{\zeta'^4( u )\, u_p u_q u_s}{( \phi + w )^2 w}
= \frac{\zeta'^2(u) ( \delta_{ps} u_q + \phi\, u_p \gamma^{q s}) }{ \phi + w  }
\end{equation}
\begin{equation} \label{eq2-19}
\gamma^{ip} \,u_p =  \frac{ u_i }{w}
\end{equation}
From \eqref{eq3-13} and \eqref{eq3-18},
\begin{equation} \label{eq2-20}
\gamma^{ik}\,\frac{\partial h_{kl}}{\partial u_s}\,\gamma^{lj} = - \frac{\zeta'^2(u) \,u_s}{ w^2 } \, a_{ij}
\end{equation}
Taking \eqref{eq2-17}--\eqref{eq2-20} into \eqref{eq2-16}, we proved the first formula.

For the second formula,
\begin{equation}  \label{eq2-36}
G_u \, = \frac{\partial F}{\partial a_{ij}} \frac{\partial a_{ij}}{\partial u} = F^{ij} \big( 2\, \frac{\partial \gamma^{ik}}{\partial u} \, h_{kl}\,\gamma^{lj} + \gamma^{ik}\,\frac{\partial h_{kl}}{\partial u}\,\gamma^{lj} \big)
\end{equation}
where
\begin{equation*}
\frac{ \partial \gamma^{ik}}{\partial u} = - \gamma^{ip}\, \frac{\partial \gamma_{pq}}{\partial u} \, \gamma^{qk}
\end{equation*}
From \eqref{eq3-11},
\begin{equation*}
\begin{aligned}
 \frac{\partial \gamma_{ik}}{\partial u} = & \, \phi' \zeta' \delta_{ik} + \frac{2 \zeta' \zeta'' u_i u_k}{\phi + w } - \frac{\zeta'^2(u) u_i u_k}{( \phi + w )^2} \big( \phi' \zeta'(u) + \frac{\phi \phi' \zeta' + \zeta'\zeta''|\nabla' u|^2}{w} \big)
 \\
= & \,\phi' \zeta' \delta_{ik} + \frac{ \zeta' u_i u_k}{\phi + w} \big( 2 \zeta'' - \frac{\zeta'}{\phi + w} \big( \phi' \zeta'  +  \frac{\phi \phi' \zeta' + \zeta'\zeta''|\nabla' u|^2}{w}   \big)\, \big) \\
= & \,\phi' \zeta' \delta_{ik} + \frac{ \zeta' u_i u_k}{\phi + w} \big( 2 \zeta'' - \frac{\zeta'^2 \zeta'' |\nabla' u|^2}{(\phi + w) w} - \frac{\phi' \zeta'^2}{w} \big) \\
= & \,\phi' \zeta' \delta_{ik} + \frac{ \zeta' u_i u_k}{\phi + w} \big(  \frac{w + \phi}{w} \zeta''  - \frac{\phi' \zeta'^2}{w} \big)
\end{aligned}
\end{equation*}
In view of \eqref{eq3-10}, the above formula becomes
\begin{equation} \label{eq2-37}
 \frac{\partial \gamma_{ik}}{\partial u} = \phi \phi' \zeta' \gamma^{ik} + \frac{\zeta' \zeta'' u_i u_k}{w}
\end{equation}
Direct calculation from \eqref{eq3-13} yields
\begin{equation} \label{eq2-38}
\frac{\partial h_{ij}}{\partial u} = ( - \frac{\phi' \zeta'^2}{w} + \frac{\phi^2 \phi' \zeta'^2}{w^3} - \frac{\phi^3 \zeta''}{w^3} ) (\nabla'_{ij} u + u \delta_{ij}) - \frac{\phi\, \zeta'}{w} \delta_{ij}
\end{equation}
Inserting \eqref{eq2-37}, \eqref{eq2-38} into \eqref{eq2-36} and in view of \eqref{eq3-18}, \eqref{eq2-19} we proved the second formula.
\end{proof}

\vspace{2mm}

\begin{cor} \label{GsGu}
Suppose that we have the $C^1$ bounds for strictly locally convex solutions $u$ of \eqref{eq2-11}:
$ u_L^K < C_0^{-1} \leq u \leq C_0$ and  $|\nabla' u| \leq C_1$ in $\overline{\Omega}$.
Then
\[| G^s | \, \leq \, C \quad \mbox{and} \quad
| G_u | \,\leq \, C ( 1 + \sum G^{ii} )
\]
\end{cor}
\begin{proof}
Note that $\{ F^{ij} (A) \}$ and $A$ can be diagonalized simultaneously by an orthonormal transformation. Consequently, the eigenvalues of the matrix $\{ F^{ij}(A)\} A$, which is not necessarily symmetric, are given by
\[ \lambda( \{ F^{ij}(A)\} A ) = ( f_1 \kappa_1, \ldots, f_n \kappa_n ) \]
In particular we have
\[ F^{ij}\, a_{ij} = \sum f_i \kappa_i \]
In addition,
for a bounded matrix $B = \{b_{ij}\}$, i.e. $|b_{ij}| \leq C$ for all $1 \leq i, j \leq n$ we have
\[ \vert b_{i k} F^{i j } a_{k j} \vert \, \leq C \sum f_i \kappa_i \]
Thus by Lemma \ref{Lemma2-1} we have
\[ \vert G^s \vert \,\leq \, C  \sum f_i \kappa_i\quad \mbox{and} \quad \vert G_u \vert \,\leq \,C ( \sum f_i \kappa_i + \sum f_i ) \]
Finally, by the concavity of $f$ and $f(0) = 0$ we can derive that $ \sum f_i \kappa_i \leq \psi \leq C$. Also, in view of \eqref{eq2-14} we have
$\sum f_i \leq C \sum G^{ii}$. Hence the corollary is proved.
\end{proof}

\vspace{5mm}

\subsection{Reformulation of equation \eqref{eq2-11} under transformation \eqref{eq6-2}}~

\vspace{3mm}

Under transformation \eqref{eq6-2}, the Dirichlet problem \eqref{eq2-11} has the following form
\begin{equation} \label{eq2-24}
\left\{ \begin{aligned}
f(\kappa[ v ])\,  = \,\,& \psi(z, v, \nabla' v)  \quad  & \mbox{in} \quad \Omega
\\
v \, =  \,\,& \varphi  \quad \quad & \mbox{on} \quad \partial \Omega
\end{aligned} \right.
\end{equation}
where $\kappa[v] = \lambda( A [ v ] )$ and $A[v] = \{ a_{ij} \}$ with $a_{ij}$ given by \eqref{eq6-1}.
Define $\mathcal{G}$ by
\[ \mathcal{G}(r, p, v) \,= \, F ( A ( r, p, v ) )\]
where $A( r, p, v )$ is obtained from $A[v]$ with $( r, p, v )$ in place of $(\nabla'^2 v, \nabla' v, v)$.
Therefore equation \eqref{eq2-24} is equivalent to
\begin{equation} \label{eq2-26}
\left\{ \begin{aligned}
\mathcal{G}(\nabla'^2 v, \nabla' v, v)\, =  \,\,& \psi(z, v, \nabla' v)   \quad  & \mbox{in} \quad \Omega
\\
v \, =  \,\,& \varphi  \quad \quad & \mbox{on} \quad \partial \Omega
\end{aligned} \right.
\end{equation}
The function $\mathcal{G}$ has similar properties as $F$. Denote
\[ \mathcal{G}^{ij}(r, p, v) = \frac{\partial \mathcal{G}}{\partial r_{ij}} (r, p, v), \quad \mathcal{G}^{i} (r, p, v) = \frac{\partial \mathcal{G}}{\partial p_{i}}(r, p, v), \quad \mathcal{G}_v (r, p, v) = \, \frac{\partial \mathcal{G}}{\partial v} (r, p, v) \]
By \eqref{eq6-1}, we can see that equation \eqref{eq2-26} is elliptic for strictly locally convex $v$, and $\mathcal{G}$ is concave with respect to $\nabla'^2 v $ for strictly locally convex $v$.

\vspace{5mm}

\section{A priori estimates}

\vspace{6mm}

In this section we derive a priori $C^2$ estimates for strictly locally convex solution $u$ to the Dirichlet problem \eqref{eq2-13} with $u \geq \underline{u}$ in $\Omega$.
\begin{equation} \label{eq4-2}
\Vert u \Vert_{C^{2}(\overline{\Omega})} \,\leq\, C
\end{equation}

The $C^1$ bound follows directly from the convexity of the radial graph $u$ with $u \geq \underline{u}$ in $\Omega$ and $u = \underline{u}$ on $\partial\Omega$. In section 4, we will derive global curvature estimates, which is equivalent to the global bound for $|\nabla'^2 u|$ on $\overline{\Omega}$ from its bound on $\partial\Omega$. Therefore in this section we focus on the boundary estimate
\begin{equation} \label{eq3-32}
| \nabla'^2 u | \leq C \quad\mbox{on} \quad \partial\Omega
\end{equation}

\vspace{3mm}

\subsection{$C^1$ estimates}

\vspace{3mm}

The $C^1$ estimate for the case $K = 0$ is established in \cite{GS93}. The method turns out to work in space forms. For the sake of completeness, we provide the proof.

\begin{lemma} \label{Lemma5-1}
Under assumption \eqref{eq1-9}, for any strictly locally convex function $u$ with $ u \geq \underline{u}$ in $\Omega$ and $u = \underline{u}$ on $\partial\Omega$ we have
\begin{equation} \label{eq5-1}
 u_L^K < C_0^{-1} \leq u \leq C_0, \quad\quad |\nabla' u| \leq C_1 \quad\mbox{in} \quad \overline{\Omega}
\end{equation}
where $C_0$ depends only on $\Omega$, $\sup_{\partial\Omega}\underline{u}$ and $\inf_{\Omega}\underline{u}$; $C_1$ depends in addition on $\sup_{\partial\Omega}\vert\nabla'\underline{u}\vert$.
\end{lemma}

\begin{proof}
Assume that $u$ achieves its maximum at $P \in \Omega$. Then there exists $Q \in \partial \Omega$ and a geodesic in $\Omega$ joining from $P$ to $Q$, with a total length $l \leq \frac{\pi}{2} - \epsilon$ for some $\epsilon > 0$. Since $u$ is strictly locally convex, i.e. $u$ satisfies \eqref{eq3-17}, we have on the geodesic
\[ u'' + u > 0 \]
if we use arc length $s$ as the parameter. It follows that
\[ \Big( \big( \frac{u}{\cos s} \big)' \cos^2 s \Big)' = ( u'' +  u ) \cos s > 0 \quad \mbox{for} \quad 0 \leq s \leq l  \]
Hence
\[ \big( \frac{u}{\cos s} \big)' \cos^2 s \geq u'(0) = 0 \]
and therefore
\[ u(P) \leq \frac{u(Q)}{\cos l} \leq \frac{\sup_{\partial \Omega} u}{\cos(\frac{\pi}{2} - \epsilon)} =  \frac{\sup_{\partial \Omega} \underline{u}}{\cos(\frac{\pi}{2} - \epsilon)}\]
A lower bound for $u$ can be seen directly from
\[u \,\geq\, \underline{u} \,\geq\, \inf_{\Omega}\underline{u} \,>\, u_L^K \quad \mbox{in} \,\,\Omega \]

For the gradient estimate, note that by \eqref{eq3-17} we have
\begin{equation} \label{eq3-26}
\Delta' \,u + n u > 0 \quad \mbox{in} \quad \Omega
\end{equation}
where $\Delta'$ is the Laplace-Beltrami operator on $\mathbb{S}^n$. Let $\bar{u}$ be the solution of
\begin{equation*}
\left\{\begin{aligned}
\Delta'\, \bar{u} + n C_0 & = 0 \quad &\mbox{in} \quad \Omega, \\
\bar{u} & = \underline{u} \quad &\mbox{on} \quad \partial\Omega
\end{aligned}\right.
\end{equation*}
By comparison principle, we have $\underline{u} \leq u \leq \overline{u} $ in $\overline{\Omega}$. Since the tangential derivatives of $u$ on $\partial\Omega$ are known, we obtain
\begin{equation} \label{eq5-5}
| \nabla' u | \leq C_1 \quad \mbox{on} \quad \partial \Omega
\end{equation}

Now we estimate the gradient $\nabla' u$ on $\overline{\Omega}$. Consider the test function
\[ w = \sqrt{u^2 + |\nabla' u|^2} \]
Assume $w$ attains its maximum at $z_0 \in \Omega$. Choose a local orthonormal frame $e_1, \ldots, e_n$ around $z_0$.
At $z_0$, there holds
\[ w w_i =  ( u_{ik} + u \,\delta_{ik} ) \,u_k = 0, \quad i = 1, \ldots, n \]
By \eqref{eq3-17} we have $\nabla' u (z_0) = 0$ and hence
\[ \sup\limits_{\overline{\Omega}} | \nabla' u | \leq w(z_0) \leq \sup\limits_{\overline{\Omega}} u  \]
We thus obtain the estimate
\begin{equation} \label{eq5-6}
| \nabla' u | \leq C_1  \quad \mbox{in} \quad \overline{\Omega}
\end{equation}
\end{proof}

\vspace{3mm}

\subsection{Boundary estimates for second derivatives}~

\vspace{3mm}

Consider any fixed point $z_0 \in \partial \Omega$. Choose a local orthonormal frame field $e_1, \ldots, e_n$ around $z_0$ on $\Omega$, which is obtained by parallel translation of a local orthonormal frame field on $\partial\Omega$ and the interior, unit, normal vector field to $\partial \Omega$, along the geodesics perpendicular to $\partial \Omega$ on $\Omega$. Assume that $e_n$ is the parallel translation of the unit normal field on $\partial \Omega$.

Since $u = \varphi$ on $\partial \Omega$,
\[ \nabla'_{\alpha\beta} (u - \varphi) = - \nabla'_n ( u - \varphi )\,\Gamma_{\alpha\beta}^n, \quad \alpha, \beta < n \quad\mbox{on}\quad \partial\Omega \]
where $\Gamma_{ij}^k$ are the Christoffel symbols of $\nabla'$ with respect to the frame $e_1, \ldots, e_n$ on $\mathbb{S}^n$. We thus obtain
\begin{equation} \label{eq3-27}
| \nabla'_{\alpha \beta} u (z_0) | \leq C, \quad \alpha, \beta < n
\end{equation}
In what follows, the Greek letters $\alpha$, $\beta$, $\ldots$ indicate the indices from $1$ to $n - 1$.

Let $\rho(z)$ and $d(z)$ denote the distances from $z \in \overline{\Omega}$ to $z_0$ and $\partial \Omega$ on $\mathbb{S}^n$, respectively.
Set
\[ \Omega_\delta = \{ z \in \Omega : \rho(z) < \delta \}  \]
Choose $\delta_0 > 0$ sufficiently small such that $\rho$ and $d$ are smooth in $\Omega_{\delta_0}$, on which, we have
\[ |\nabla' d| = 1, \quad\quad - C\, I \leq \nabla'^2\, d \leq C \,I, \quad
\quad |\nabla' \rho| = 1, \quad\quad I \,\leq \,\nabla'^2 \,\rho^2 \,\leq \, 3 I \]
where $C$ depends only on $\delta_0$ and the geometric quantities of $\partial \Omega$, and
\[ \nabla'^2 \underline{u} + \underline{u} \, I \geq 4 \, c_0 \, I  \]
for some constant $c_0 > 0$ because of the strict local convexity  of $\underline{u}$ near $\partial\Omega$.

We will need the following barrier function
\[ \Psi = A v + B \rho^2 \]
with
\[ v = u - \underline{u} + \epsilon\, d - \frac{N}{2} \,d^2 \]
and the linearized operator associated with equation \eqref{eq2-13}
\begin{equation} \label{eq2-21}
 L  = G^{ij}\, \nabla'_{i j} + ( G^i - \psi^i ) \,\nabla'_i
\end{equation}
to estimate the mixed tangential normal and pure normal second derivatives at $z_0$.
By direct calculation and Corollary \ref{GsGu} we have

\begin{equation} \label{eq3-20}
\begin{aligned}
 L v = \,\,& \Big( G^{ij} \nabla'_{ij} + (G^i - \psi^i) \nabla'_i \Big) \Big( u - \underline{u} + \epsilon \,d - \frac{N}{2} d^2 \Big) \\
     = \,\, & G^{ij} \nabla'_{ij} \Big( u - \underline{u} - \frac{N}{2} \,d^2 \Big) + \epsilon \,G^{ij} \nabla'_{ij}  d + (G^i - \psi^i) \nabla'_i \Big( u - \underline{u} + \epsilon\, d - \frac{N}{2} \,d^2 \Big) \\
     \leq \,\, & G^{ij} \Big( \nabla'_{ij} u - \,\big( \nabla'_{ij}(\underline{u} + \frac{N}{2} \,d^2 ) - 2 c_0 \delta_{ij} \big) \Big) \\ & - 2 c_0 \sum G^{ii}
       + C \epsilon \sum G^{ii} + C  ( 1 +  \epsilon +  N \delta )
\end{aligned}
\end{equation}
Since $G(\nabla'^2 u, \nabla' u, u)$ is concave  with respect to $\nabla'^2 u$,
\begin{equation} \label{eq3-21}
\begin{aligned}
& G^{ij} \Big( \nabla'_{ij} u -\, \big( \nabla'_{ij} (\underline{u} + \frac{N}{2} \,d^2 ) - 2 c_0 \delta_{ij}\big) \Big) \\ \leq &\,\, G( \nabla'^2 u, \nabla' u, u ) - G\Big( \nabla'^2\big( \underline{u} + \frac{N}{2} \,d^2 \big) - 2 c_0 I, \nabla' u, u \Big)
\end{aligned}
\end{equation}
Note that
\begin{equation*}
\begin{aligned}
&  \nabla'^2\big( \underline{u} + \frac{N}{2} \,d^2 \big ) - 2 c_0 I + u I \\
= \quad &  \nabla'^2 \underline{u} + \underline{u}\, I +  N d \nabla'^2 d + N \nabla' d \otimes \nabla' d - 2 c_0 I + ( u - \underline{u} ) I \\
\geq \quad & 2 c_0 I - C N \delta I + N \nabla' d \otimes \nabla' d := \mathcal{H}
\end{aligned}
\end{equation*}
Denote $g^{-\frac{1}{2}} = \{\gamma^{ik}\}$.
We thus have
\begin{equation} \label{eq3-22}
\begin{aligned}
& G\Big( \nabla'^2\big( \underline{u} + \frac{N}{2} \,d^2 \big) - 2 c_0 I, \nabla' u, u \Big)  \\
= \,\, & F \Big( \frac{- \phi \,\zeta'(u)}{\sqrt{\phi^2 + \zeta'^2 |\nabla' u|^2 }}\, g^{-1/2} \big( \nabla'^2( \underline{u} + \frac{N}{2} \,d^2 ) - 2 c_0 I + u I \big)\,\, g^{-1/2} \Big) \\
\geq  &\,\, F \Big( \frac{- \phi\, \zeta'(u)}{\sqrt{\phi^2 + \zeta'^2 |\nabla' u|^2 }} \,g^{-1/2} \mathcal{H}\, g^{-1/2} \Big)
\\ =  & \,\,F \Big( \frac{- \phi\, \zeta'(u)}{\sqrt{\phi^2 + \zeta'^2 |\nabla' u|^2 }} \,\mathcal{H}^{1/2} \,g^{-1} \,\mathcal{H}^{1/2} \Big) \\
\geq & \,\,F \Big( \frac{- \phi\, \zeta'(u)}{\sqrt{\phi^2 + \zeta'^2 |\nabla' u|^2 }}\,\mathcal{H}^{1/2} \, \frac{1}{ \phi^2 + \zeta'^2(u) |\nabla' u|^2} \,I \,\mathcal{H}^{1/2} \Big) \\
=  & \,\,F \Big( \frac{- \phi\, \zeta'(u)}{(\phi^2 + \zeta'^2 |\nabla' u|^2 )^{3/2}} \,\mathcal{H} \Big)\,
\geq \, \, F ( \tilde{c}\, \mathcal{H} )
\end{aligned}
\end{equation}
where $\tilde{c}$ is a positive constant depending only on $C_0$ and $C_1$.
Combining \eqref{eq3-20}--\eqref{eq3-22} we have
\begin{equation} \label{eq3-23}
 L v \leq   \, - F ( \tilde{c}\,\mathcal{H} )   + (   C \epsilon  - 2 c_0 ) \,\sum G^{ii} + C ( 1 + \epsilon + N \delta )
\end{equation}
where $\mathcal{H} = \mbox{diag} \Big( 2 c_0 - C N \delta,\,\, \ldots,\,\, 2 c_0 - C N \delta,\,\, 2 c_0 - C N \delta + N\Big)$.
Choose $N$ sufficiently large and $\epsilon$, $\delta$ sufficiently small with $\delta$ depending on $N$ such that
\[  C \epsilon  \leq c_0, \quad  C N \delta \leq c_0, \quad  -F ( \tilde{c}\,\mathcal{H} )   + C  + 2 c_0  \leq - 1, \]
Therefore, \eqref{eq3-23} becomes
\begin{equation} \label{eq3-25}
 L v \leq - c_0 \sum G^{ii} - 1
\end{equation}
We then choose $\delta \leq \frac{2 \epsilon}{N}$ such that
\[ v \,\geq \, 0 \quad \mbox{in} \quad \Omega_\delta \]

A direct consequence of \eqref{eq3-25} is
\begin{equation} \label{eq3-29}
L \Psi \, = \, A \,L v \,+\, B\,L( \rho^2 ) \,\leq \,A ( - c_0 \sum G^{ii} - 1 ) + B C ( 1 + \sum G^{ii} ) \quad\mbox{in}\,\,\Omega_{\delta}
\end{equation}
which will be used later.  Besides, we also need to estimate $L ( \nabla'_k u )$. For this, we first apply the formula
\begin{equation*}
\nabla'_{ij} ( \nabla'_k u ) = \,\nabla'_k \nabla'_{ij} u + \Gamma^{l}_{ik} \nabla'_{jl} u + \Gamma_{jk}^l \nabla'_{il} u + \nabla'_{k} \Gamma_{ij}^l \,u_l
\end{equation*}
to obtain
\begin{equation} \label{eq3-30}
\begin{aligned}
L ( \nabla'_k u ) =\,& G^{ij} \nabla'_{ij} ( \nabla'_k u ) + ( G^i - \psi^i ) \nabla'_i (\nabla'_k u) \\
=  \,& \big( G^{ij}\nabla'_k \nabla'_{ij} u + ( G^i - \psi^i ) \nabla'_{ik} u \big) \\ & + G^{ij} \Gamma^{l}_{ik} \nabla'_{jl} u + G^{ij} \Gamma_{jk}^l \nabla'_{il} u + G^{ij} \nabla'_{k} \Gamma_{ij}^l \,u_l   + ( G^i - \psi^i ) \Gamma_{ik}^l \,u_l
\end{aligned}
\end{equation}
By \eqref{eq2-14} and \eqref{eq3-18} we have
\begin{equation*}
G^{ij} \Gamma^{l}_{ik} ( \nabla'_{jl} u + \,u \,\delta_{jl}) =\, F^{st} \gamma^{is} \gamma^{jt} \Gamma^{l}_{ik}  \cdot  \gamma_{jp}\, a_{pq}\,\gamma_{ql}\\
=\,( \gamma^{is} \Gamma^{l}_{ik} \gamma_{q l } ) \,F^{st} \,a_{t q}
\end{equation*}
The term $G^{ij} \Gamma_{jk}^l \nabla'_{il} u$ can be computed similarly. Taking the covariant derivative of \eqref{eq2-13} and applying Corollary \ref{GsGu} we have
\[ | G^{ij}\nabla'_k \nabla'_{ij} u +  ( G^i - \psi^i ) \, \nabla'_{k i} u | \leq C + | (\psi_u - G_u) u_k | \leq C (1 + \sum G^{ii}) \]
From all these above, \eqref{eq3-30} can be estimated as
\begin{equation} \label{eq3-19}
\vert L ( \nabla'_k u ) \vert \leq  \, C ( 1 + \sum G^{ii} )
\end{equation}

For fixed $\alpha < n$, choose $B$ sufficiently large such that
\[ \Psi \pm \nabla'_\alpha (u - \varphi) \geq 0 \quad \mbox{on} \quad \partial\Omega_\delta   \]
From  \eqref{eq3-29} and \eqref{eq3-19}
\[ L ( \Psi \pm \nabla'_\alpha (u - \varphi) ) \leq A ( - c_0 \sum G^{ii} - 1 ) + B C ( 1 + \sum G^{ii} )  \]
Then choose $A$ sufficiently large such that
\[ L ( \Psi \pm \nabla'_\alpha (u - \varphi) ) \leq 0 \quad \mbox{in} \quad \Omega_\delta \]
Applying the maximum principle we have
\[ \Psi \pm \nabla'_\alpha (u - \varphi) \geq 0  \quad \mbox{in} \quad \Omega_\delta \]
which implies
\begin{equation} \label{eq3-28}
\vert \nabla'_{\alpha n} u (z_0) \vert \leq C
\end{equation}

It remains to estimate the double normal derivative $\nabla'_{n n} u $ on $\partial\Omega$. Since $\sigma_1( \kappa[u] ) > 0$, it suffices to derive an upper bound
\[ \nabla'_{n n} u  \leq C \quad \mbox{on} \quad \partial\Omega \]
In \cite{Cruz}, the author gives a proof (see also \cite{Tru, Guan99}). For the sake of consistency and also to show some details, we provide a complete proof which is slightly different.
Let $\kappa' = (\kappa'_1, \ldots, \kappa'_{n-1})$ be the roots of
\[ \det (  \kappa' \, g_{\alpha\beta} - h_{\alpha \beta}  ) = 0, \quad \quad \alpha,\, \beta < n \]
By definition of $\Gamma_k$, we can verify that the projection of $\Gamma_k \subset \mathbb{R}^n$ onto $\mathbb{R}^{n - 1}$ is exactly
\[ \Gamma'_{k - 1} : = \,\{ \lambda' = (\lambda_1, \ldots, \lambda_{n-1}) \in \mathbb{R}^{n - 1} \,| \, \sigma_{j} (\lambda' ) > 0,\,\, \, j = 1, \ldots, k - 1 \} \]
Since $u$ is $k$-admissible, i.e., $\kappa [u] \in \Gamma_k$, it follows that $\kappa'[u] \in \Gamma'_{k - 1}$.  Note that $\kappa'[u]$ may not be $(\kappa_1, \ldots, \kappa_{n-1})[u]$.
For $z \in \partial \Omega$, define
\[ \tilde{d} (z) \,: =\, \frac{w}{- \zeta' \,\phi } \,\,\mbox{dist} ( \kappa'[u](z), \,\partial \Gamma'_{k - 1} ) \]
where $w = \sqrt{\phi^2 + \zeta'^2 |\nabla' u|^2}$. We want to prove that
$ \tilde{d} (z_1) := \min\limits_{z \in \partial\Omega}\,\tilde{d} (z)$ has a positive uniform lower bound.

Let $\tau_1, \ldots, \tau_{n-1}, e_n$ be a local frame field around $z_1$ on $\Omega$, obtained by parallel translation of a local frame field $\tau_1, \ldots, \tau_{n-1}$ around $z_1$ on $\partial\Omega$ satisfying
\[ g_{\alpha \beta} = \delta_{\alpha\beta}, \quad\quad h_{\alpha \beta}(z_1) = \kappa'_{\alpha}(z_1) \, \delta_{\alpha\beta}, \quad\quad \kappa'_1 (z_1) \leq \ldots \leq \kappa'_{n-1} (z_1) \]
and the interior, unit, normal vector field $e_n$ to $\partial \Omega$, along the geodesics perpendicular to $\partial \Omega$ on $\Omega$. Note that  $\tau_1, \ldots, \tau_{n - 1}$ depend on $\varphi$ and $\nabla'_{e_1} \varphi, \ldots, \nabla'_{e_{n - 1}} \varphi$ on $\partial \Omega$. This can be seen if we let $\tau_{\alpha} = \sum_{i < n} \eta_{\alpha i} \,e_i$ and observe that
\[ \eta_{\alpha i} \,\big( \phi^2 \,\delta_{ij} + \zeta'^2(u)\, \nabla'_{e_i} u \, \nabla'_{e_j} u \big) \,\eta_{\beta j } = \delta_{\alpha \beta} \quad\quad\mbox{on} \quad\partial\Omega \]
which implies that all elements of the invertible matrix $\eta = \{ \eta_{\alpha i} \}$ depend only on $\varphi$ and $\nabla'_{e_1} \varphi, \ldots, \nabla'_{e_{n - 1}} \varphi$ on $\partial \Omega$.

By Lemma 6.1 of \cite{CNSIII}, there exists $\gamma' = (\gamma_1, \ldots, \gamma_{n-1}) \in \mathbb{R}^{n - 1}$ with $\gamma_1 \geq \ldots \geq \gamma_{n - 1} \geq 0$ and $\sum \gamma_{\alpha}^2 = 1$ such that $\Gamma'_{k - 1} \subset \{ \lambda' \in \mathbb{R}^{n-1} | \,\gamma' \cdot \lambda' > 0 \}$ and
\begin{equation} \label{eq3-35}
\tilde{d}(z_1) =   \,\frac{w}{- \zeta' \,\phi } \, \sum\limits_{\alpha < n} \gamma_{\alpha} \,\kappa'_{\alpha} (z_1) =  \,\sum\limits_{\alpha < n} \,\gamma_{\alpha} \,\big( \nabla'_{\alpha \alpha} u + u \,\sigma_{\alpha \alpha} \big) (z_1)
\end{equation}
Note that $\gamma'$ depends on $u$ and $\sum \gamma_{\alpha} \geq 1$.

Since $\underline{u}$ is strictly locally convex near $\partial\Omega$,
\[\sum\limits_{\alpha < n}  \gamma_{\alpha} \big( \nabla'_{\alpha \alpha}  \underline{u} + \underline{u} \,\sigma_{\alpha\alpha} \big) (z_1)  \,\geq \,  2 \, c_1  \]
where $c_1$ is a uniform positive constant.  Hence,
\[\begin{aligned}
& \nabla'_n ( u - \underline{u} ) (z_1) \sum\limits_{\alpha < n} \gamma_{\alpha} \tilde{\Gamma}_{\alpha\alpha}^n (z_1) = \sum\limits_{\alpha < n} \gamma_{\alpha} \nabla'_{\alpha \alpha} ( \underline{u} - u ) (z_1) \\ = \,\, & \sum\limits_{\alpha < n}  \gamma_{\alpha} \big( \nabla'_{\alpha \alpha}  \underline{u} + \underline{u} \,\sigma_{\alpha\alpha} \big) (z_1) - \sum\limits_{\alpha < n}  \gamma_{\alpha} \big( \nabla'_{\alpha \alpha}  u + u\, \sigma_{\alpha\alpha} \big) (z_1) \\
\geq  \,\,& 2 \, c_1 - \,\tilde{d}(z_1)
\end{aligned}
\]
where $\tilde{\Gamma}_{ij}^k$ are the Christoffel symbols of $\nabla'$ with respect to the local frame $\tau_1, \ldots, \tau_{n-1}, e_n$ on $\mathbb{S}^n$.
We may assume
$\tilde{d} (z_1) \leq  \, c_1$, for, otherwise we are done. Then
\[ \nabla'_n ( u - \underline{u} ) (z_1) \sum\limits_{\alpha < n} \gamma_{\alpha} \tilde{\Gamma}_{\alpha\alpha}^n (z_1)
\geq   \, c_1 \]
Note that  $0 <  \nabla'_n( u - \underline{u} ) (z_1) \leq C$. Thus,
\[ \sum\limits_{\alpha < n} \gamma_{\alpha} \tilde{\Gamma}_{\alpha\alpha}^n (z_1)
\geq   \,2 \, c_2 > 0 \]
A straightforward calculation shows that
\begin{equation} \label{eq3-36}
\tilde{\Gamma}_{\alpha \beta}^n = \sum_{i, j  < n} \eta_{\alpha i}\, \eta_{\beta j} \,\Gamma_{ij}^n, \quad\quad \alpha, \,\beta < n
\end{equation}
 Thus by continuity of $\tilde{\Gamma}_{\alpha\alpha}^n (z)$ and $0 \leq \gamma_{\alpha} \leq 1$,
\begin{equation} \label{eq3-33}
 \sum\limits_{\alpha < n} \gamma_{\alpha} \tilde{\Gamma}_{\alpha\alpha}^n (z)  >  \,\sum\limits_{\alpha < n} \gamma_{\alpha} \tilde{\Gamma}_{\alpha\alpha}^n (z_1) - c_2 \geq c_2
\end{equation}
on $\Omega_\delta = \{ z \in \Omega \,| \,\mbox{dist}_{\mathbb{S}^n}(z_1, z) < \delta  \}$
for some small uniform $\delta > 0$ and a uniform positive constant $c_2$.

On the other hand, by Lemma 6.2 of \cite{CNSIII}, for any $z \in \partial \Omega$ near $z_1$,
\[ \,\sum\limits_{\alpha < n} \,\gamma_{\alpha} \,\big( \nabla'_{\alpha \alpha} u + u \,\sigma_{\alpha \alpha} \big) (z) \,\geq \, \frac{w}{- \zeta' \,\phi } \,\sum\limits_{\alpha < n} \,\gamma_{\alpha} \,\kappa'_{\alpha} [ u ] (z) \,\geq \,\tilde{d} (z)\, \geq\, \tilde{d} (z_1) \]
and consequently,
\begin{equation}  \label{eq3-34}
\begin{aligned}
& \nabla'_n ( u - \varphi ) (z) \sum\limits_{\alpha < n} \gamma_{\alpha} \tilde{\Gamma}_{\alpha\alpha}^n (z) = \sum\limits_{\alpha < n} \gamma_{\alpha} \nabla'_{\alpha \alpha} ( \varphi - u ) (z) \\ = \,\, & \sum\limits_{\alpha < n}  \gamma_{\alpha} \big( \nabla'_{\alpha \alpha}  \varphi + \varphi \,\sigma_{\alpha\alpha} \big) (z) - \sum\limits_{\alpha < n}  \gamma_{\alpha} \big( \nabla'_{\alpha \alpha}  u + u\, \sigma_{\alpha\alpha} \big) (z) \\
\leq  \,\,& \sum\limits_{\alpha < n}  \gamma_{\alpha} \big( \nabla'_{\alpha \alpha}  \varphi + \varphi \,\sigma_{\alpha\alpha} \big) (z) -  \,\tilde{d}(z_1)
\end{aligned}
\end{equation}
In view of \eqref{eq3-33}, we define in $\Omega_{\delta}$,
\[ \Phi \, =  \,  \frac{1}{\sum\limits_{\alpha < n} \gamma_{\alpha} \tilde{\Gamma}_{\alpha\alpha}^n } \,\left(\sum\limits_{\alpha < n}  \gamma_{\alpha} \big( \nabla'_{\alpha \alpha}  \varphi + \varphi \,\sigma_{\alpha\alpha} \big)  -  \,\tilde{d}(z_1)\right)   - \nabla'_n ( u - \varphi ) \]
By \eqref{eq3-34},  $\Phi \geq 0$ on $\partial\Omega \cap \overline{\Omega_\delta}$.
In view of \eqref{eq3-19} and \eqref{eq3-36}, we have
\[  L(\Phi) \leq   C \,( 1 + \sum G^{ii} ) \]
Now choose $B$ large such that $\Psi + \Phi \geq 0$ on $\partial \Omega_\delta$. In view of \eqref{eq3-29},  we then choose $A$ sufficiently large such that
$L(\Psi + \Phi) \leq 0$ in $\Omega_{\delta}$. By \eqref{eq3-35}, $(\Psi + \Phi) (z_1) = 0$. It follows that $\nabla'_n (\Psi + \Phi) (z_1) \geq 0$ and hence
\[ \nabla'_{nn} u (z_1) \leq C. \]
Along with \eqref{eq3-27} and \eqref{eq3-28}, we thus have a bound $|\nabla'^2 u (z_1)| \leq C$, equivalently by \eqref{eq3-18}, a bound for all the principal curvatures of the radial graph  at $z_1$. By \eqref{eq1-7},
\[ \mbox{dist} ( \kappa[u](z_1), \partial\Gamma_k )  \geq c_3  \]
and consequently on $\partial \Omega$,
\[ \tilde{d}(z) \geq \tilde{d}(z_1) = \frac{w}{- \zeta' \,\phi }\,\mbox{dist} ( \kappa'[u](z_1), \partial\Gamma'_{k - 1} )   \geq c_4  \]
where $c_3$ and $c_4$ are positive uniform constants.
By a proof similar to Lemma 1.2 of \cite{CNSIII}, we know that there exists $R > 0$ depending on the bounds in \eqref{eq3-27} and \eqref{eq3-28} such that if $\nabla'_{nn}u(z_0) \geq R$ and $z_0 \in \partial \Omega$, then the principal curvatures $(\kappa_1, \ldots, \kappa_n)$ at $z_0$ satisfy
\[ \kappa_{\alpha} = \kappa'_{\alpha} + o(1), \quad\quad \alpha < n \]

\[ \kappa_n = \frac{ h_{nn} - g_{1n} h_{n1} - \ldots - g_{n n-1} h_{n n-1} }{g_{n n} - g_{1 n}^2 - \ldots - g_{n n-1}^2} \left( 1 + \mathcal{O} \Big( \frac{g_{n n} - g_{1 n}^2 - \ldots - g_{n n-1}^2}{ h_{nn} - g_{1n} h_{n1} - \ldots - g_{n n-1} h_{n n-1} } \Big) \right) \]
in the local frame $\tau_1, \ldots, \tau_{n-1}, e_n$  around $z_0$.   However, when $R$ is sufficiently large,
\[ G( \nabla'^2 u, \nabla' u, u ) (z_0) = f( \kappa [u] ) (z_0) \, > \, \psi(z_0, u, \nabla'u) \]
contradicting with equation \eqref{eq2-13}. Hence $\nabla'_{nn} u \leq C$ on $\partial\Omega$ and therefore we proved \eqref{eq3-32}.

\vspace{5mm}

\section{Global curvature estimates}

\vspace{4mm}

Our main result on global curvature estimates can be stated as follows. The following proof is motivated by the work \cite{GRW15, SX15}.

\begin{thm} \label{Theorem 1}
Let $ \Sigma = \{ ( z, \rho (z) ) \,\vert\, z \in \Omega \subset \mathbb{S}^n \}$ be a strictly locally convex $C^4$ hypersurface in $N^{n + 1}(K)$ satisfying \eqref{eq1-0} for some positive function $\psi(V, \nu) \in C^2 (\Gamma)$, where $\Gamma$ is an open neighborhood of the unit normal bundle of $\Sigma$ in $ N^{n+1}(K) \times \mathbb{S}^n$. Suppose
\[ 0 < C_0^{-1} \leq \rho(z) \leq C_0 < \rho_U^K \quad\quad \mbox{and}\quad\quad |\nabla' \rho| \leq C_1\quad \mbox{on} \quad \overline{\Omega}\]
where $C_0$ and $C_1$ are positive constants.
Then there exists a positive constant $C$ depending only on $n$, $k$, $C_0$, $C_1$, $\inf \psi$ and $\Vert\psi\Vert_{C^2}$ such that
\[ \sup\limits_{\substack{ z \in \Omega \\  i = 1, \ldots, n}}  \kappa_i (z)  \leq C \,( 1 +  \sup\limits_{\substack{z \in \partial \Omega\\ i = 1, \ldots, n}}  \kappa_i (z) )  \]
\end{thm}

\begin{proof}
It suffices to estimate from above for the largest principal curvature $\kappa_{\max} = \max_{1 \leq i \leq n} \kappa_i$ of $\Sigma$.
To construct a test function, we will make use of the following ingredients:
\[  \Phi(\rho) = \int_0^{\rho} \phi(r) \,d r \]
and the support function
\[ \tau = \bar{g}( V,\, \nu ) = \langle V,\, \nu \rangle \,=\,\Big\langle \phi(\rho) \frac{\partial}{\partial \rho}, \frac{- \nabla' \rho + \phi^2 \,\frac{\partial}{\partial \rho}}{\sqrt{\phi^4 + \phi^2 |\nabla' \rho|^2}} \Big\rangle   \]
Note that $\tau$ has a positive lower bound.
Now define the test function
\begin{equation} \label{eq2-2}
\Theta = \frac{1}{2} \ln P(\kappa) - N \,\ln \tau + \beta\, \Phi
\end{equation}
where
\[ P( \kappa ) = \kappa_1^2 + \cdots + \kappa_n^2,\]
$\beta = u_L^K$ and $N$ is a positive constant to be determined later.

Assume that $\Theta$ achieves its maximum value at ${\bf x}_0 = (z_0, \rho(z_0))\in \Sigma$.
Choose a local orthonormal frame $E_1, \ldots, E_n$ around ${\bf x}_0$ such that $h_{ij}({\bf x}_0) = \kappa_i \,\delta_{ij}$, where $\kappa_1, \ldots, \kappa_n$ are the principal curvatures of $\Sigma$ at ${\bf x}_0$ with $\kappa_1 \geq \ldots \geq \kappa_n > 0$.  Let $\nabla$ denote the Levi-Civita connection on $\Sigma$ with respect to the metric $g$. Then, at ${\bf x}_0$,
\begin{equation} \label{eq2-3}
\frac{1}{P} \,\sum\limits_l \kappa_l \,h_{lli} - N \, \frac{{\tau}_i}{\tau} + \beta \,\Phi_i = 0
\end{equation}

\begin{equation} \label{eq2-4}
\frac{1}{P} (\sum\limits_{pq} h_{pqi}^2 + \sum\limits_l \kappa_l h_{llii}) - \frac{2}{P^2}(\sum\limits_l \kappa_l h_{lli})^2 - N\, \frac{{\tau}_{ii}}{\tau} + N \,\frac{{\tau}_i^2}{{\tau}^2} \,+ \beta \,\Phi_{ii}\, \leq 0
\end{equation}

In space forms, the Codazzi equation is
\begin{equation} \label{eq2-9}
\nabla_l h_{ij} = \nabla_j h_{il}
\end{equation}
and by Gauss equation we have
\begin{equation} \label{eq2-5}
 h_{iill} = h_{llii} + \kappa_l \kappa_i^2 - \kappa_l^2 \kappa_i + K (\kappa_i - \kappa_l)
\end{equation}

Covariantly Differentiating \eqref{eq1-0} twice yields

\begin{equation} \label{eq2-6}
 \sigma_k^{ii} h_{iil} = \phi'\, d_V\psi (E_l) + \kappa_l \,d_\nu\psi (E_l)
\end{equation}

\begin{equation} \label{eq2-7}
\sigma_k^{ii} h_{iill} + \sigma_k^{pq,\, rs} h_{pql} h_{rsl} \geq - C - C \kappa_l^2 + \sum\limits_m h_{mll} \,d_{\nu}\psi (E_m)
\end{equation}
Note that  we have used the property of the conformal Killing field $V$
\[ \nabla_{E_l} V = \phi' \,E_l  \]

By \eqref{eq2-3}, \eqref{eq2-4}, \eqref{eq2-9}, \eqref{eq2-5}, \eqref{eq2-7}  as well as $\sigma_k^{ii} \kappa_i = k \,\psi$ and
\begin{equation*}
 - \sigma_k^{pq, rs} h_{pql} h_{rsl} = - \sigma_k^{pp, qq} h_{ppl} h_{qql} + \sigma_k^{pp, qq} h_{pql}^2
\end{equation*}
we have
\begin{equation} \label{eq2-8}
\begin{aligned}
& \frac{1}{P} \sum\limits_{i p q} \sigma_k^{ii} h_{pqi}^2 - \frac{2}{P^2} \sum\limits_i \sigma_k^{ii} (\sum\limits_l \kappa_l h_{lli})^2
  - \frac{1}{P} \sum\limits_{pql} \kappa_l \sigma_k^{pp, qq} h_{ppl} h_{qql} + \frac{1}{P} \sum\limits_{pql} \kappa_l \sigma_k^{pp, qq} h_{pql}^2 \\
& - \sum\limits_i \sigma_k^{ii} \kappa_i^2 - \frac{N}{\tau} \sum\limits_i \sigma_k^{ii} {\tau}_{ii} + \frac{N}{{\tau}^2} \sum\limits_i \sigma_k^{ii} {\tau}_i^2 + \beta
   \sigma_k^{ii}\,\Phi_{ii} \\
& + K \, \sum\limits_i \sigma_k^{ii} + \frac{N}{\tau} \sum\limits_m {\tau}_m d_{\nu}\psi (E_m) - \beta\,\sum\limits_m \Phi_m\,d_{\nu} \psi (E_m) - \frac{C}{P} \sum\limits_l \kappa_l - \frac{C}{P} \sum\limits_l \kappa_l^3 \leq 0
\end{aligned}
\end{equation}
Applying \eqref{eq2-6} as well as the following equations which can be derived by straight forward calculation (see Lemma 2.2 and Lemma 2.6 in \cite{GL13} for the proof)
\[ \Phi_i = \phi(\rho) \,\rho_i, \quad\quad \Phi_{ii} = \phi' - \tau \,\kappa_i\]
\[ {\tau}_i = \phi(\rho)\, \rho_i\, \kappa_i \]
\[ {\tau}_{ii} = \phi(\rho)\,\sum\limits_m \rho_m\, h_{iim} + \phi'(\rho)\,\kappa_i - \tau \,\kappa_i^2 \]
\eqref{eq2-8} becomes

\begin{equation} \label{eq2-10}
\begin{aligned}
& \frac{1}{P} \sum\limits_{i p q} \sigma_k^{ii} h_{pqi}^2
   -\frac{2}{P^2} \sum\limits_i \sigma_k^{ii} (\sum\limits_l \kappa_l h_{lli})^2
   - \frac{1}{P} \sum\limits_{pql} \kappa_l \sigma_k^{pp, qq} h_{ppl} h_{qql} + \frac{1}{P} \sum\limits_{pql} \kappa_l \sigma_k^{pp, qq} h_{pql}^2 \\
& + (N - 1) \sum\limits_i \sigma_k^{ii} \kappa_i^2 + N \frac{\phi^2}{{\tau}^2} \sum\limits_i \sigma_k^{ii} \rho_i^2 \kappa_i^2 - N \frac{k \psi \phi'}{\tau} - \beta \,\tau\, k \, \psi + ( \beta \phi' + K ) \sum\limits_i \sigma_k^{ii} \\
& - N \frac{\phi \,\phi'}{\tau} \sum\limits_m \rho_m d_{V}\psi (E_m) - \beta \,\phi(\rho)\,\sum\limits_m \rho_m \,d_{\nu}\psi(E_m) - \frac{C}{P} \sum\limits_l \kappa_l - \frac{C}{P} \sum\limits_l \kappa_l^3 \leq 0
\end{aligned}
\end{equation}
Now we apply a result from \cite{GRW15} (see Lemma 2.2 and Corollary 4.4 in \cite{GRW15}) for tackling third order derivatives.

\begin{lemma} \label{GRW}
There exists a positive constant $A$ and a finite sequence of positive numbers $\{ \delta_i \}_{i = 1}^k$ such that if the inequality
$\kappa_i/\kappa_1 \leq \delta_i$ holds for some $1 \leq i \leq k$, then
\begin{equation*}
0 \leq \frac{1}{P} \big[ \sum\limits_l \kappa_l ( A (\sigma_k)_l^2 - \sigma_k^{pp, qq} h_{ppl} h_{qql} + \sigma_k^{pp, qq} h_{pql}^2 ) + \sum\limits_{ipq} \sigma_k^{ii} h_{pqi}^2 \big] - \frac{2}{P^2} \sum\limits_i \sigma_k^{ii}(\sum\limits_l \kappa_l h_{lli})^2
\end{equation*}
\end{lemma}

Let $A$ and $\{ \delta_i \}_{i = 1}^k$ be given as in Lemma \ref{GRW}.
We divide our discussion into two cases.

Case (i):  If there exists some $2 \leq i \leq k$ such that $\kappa_i \leq \delta_i \,\kappa_1$, by \eqref{eq2-6} and Lemma \ref{GRW}, \eqref{eq2-10} reduces to

\[ (N - 1)\,\sum\limits_i \sigma_k^{ii} \kappa_i^2 - C N - C \beta - \frac{C (A + 1)}{P} \sum\limits_l \kappa_l - \frac{C (A + 1)}{P} \sum\limits_l \kappa_l^3 \,\leq 0 \]
Here we have used the fact that the support function $\tau$ has a positive lower bound.

Note that
\[ \sigma_k^{11} \,\kappa_1 \,\geq\, \frac{k}{n} \,\sigma_k  \]
It follows that
\[ ((N - 1)\, \frac{k}{n}\, \psi - C)\,\kappa_1 \leq C\,N \]
Choose $N$ sufficiently large we obtain $\kappa_1 \leq C(N)$.

Case (ii): If case (i) does not hold, which means $\kappa_k \geq \delta_k\,\kappa_1$, then
\[ \sigma_k \geq \kappa_1 \kappa_2 \cdots \kappa_k \geq \delta_k^k \kappa_1^k\]
and an upper bound of $\kappa_1$ follows.
\end{proof}

\begin{rem}
After the proof of Theorem \ref{Theorem 1}, the author noticed \cite{CLW18} for closed hypersurfaces in warped product spaces, where global curvature estimates for convex hypersurfaces are also derived. Though our test function appears the same as \cite{CLW18}, the choice of the coefficients is different. In space forms, the Gauss equation and Codazzi equation are simpler and hence $\beta$ can be chosen to be a fixed number depending on the sectional curvature $K$. In particular, when $K \geq 0$, $\beta$ can be zero. In \cite{CLW18}, $N$ is chosen to be large, and $\beta$ is chosen to be further large.
\end{rem}

\vspace{6mm}

\section{Existence in $\mathbb{R}^{n+1}$ and $\mathbb{H}^{n+1}$}

\vspace{4mm}

In this section and the next section, we confine ourselves to prescribed Gauss curvature equation (the case when $k = n$).
We will use classical continuity method and degree theory developed by Y. Y. Li \cite{Li89} to prove the existence of solution to the Dirichlet problem \eqref{eq2-26}.

Under the transformation $\overline{\rho} = \zeta ( \underline{u} )$ and $\underline{u} = \eta ( \underline{v} )$, the subsolution condition \eqref{eq1-4} becomes
\begin{equation} \label{eq2-29}
\left\{ \begin{aligned}
\mathcal{G}(\nabla'^2\underline{v}, \nabla' \underline{v}, \underline{v}) \,\, \geq & \,\,\psi(z,  \underline{v},  \nabla'\underline{v})  \quad  & \mbox{in} \quad  \Omega\\
\underline{v} \, = & \,\, \varphi  \quad \quad & \mbox{on} \quad  \partial \Omega \end{aligned} \right.
\end{equation}
For convenience, denote $\mathcal{G} [ v] \, = \, \mathcal{G}(\nabla'^2 v, \nabla' v,  v)$.
Consider the following two auxiliary equations.
\begin{equation} \label{eq6-8}
\left\{ \begin{aligned} \mathcal{G} [ v ] \,\, =  & \,\, \Big( ( 1 - t ) \frac{\mathcal{G}[\underline{v}]}{ \xi (\underline{v}) } + t \,\epsilon  \Big)  \,\xi (v) \quad\quad & \mbox{in} \quad \Omega \\
v \,\,  = & \,\, \underline{v} \quad \quad & \mbox{on} \quad \partial\Omega \end{aligned} \right.
\end{equation}
and
\begin{equation} \label{eq6-9}
\left\{ \begin{aligned} \mathcal{G} [ v ] \,\, =  & \,\,( 1 - t ) \,\epsilon \,\xi (v) +  t \, \psi(z, v, \nabla' v) \quad\quad & \mbox{in} \quad \Omega \\
 v \,\,  = & \,\, \underline{v} \quad \quad & \mbox{on} \quad \partial\Omega \end{aligned} \right.
\end{equation}
where $t \in [0, 1]$, $\epsilon$ is a small positive constant such that
\begin{equation} \label{eq6-20}
\mathcal{G} [\underline{v}]  \, > \, \epsilon \,\,\xi(\underline{v}) \quad\mbox{in}\quad \Omega
\end{equation}
and $\xi(v) = e^{2v}$ if $K = 0$ while $\xi(v) = \sinh v$ if $K = - 1$.

The existence result in $\mathbb{R}^{n+1}$ was given in \cite{Su16} where the author assumed the existence of a strict subsolution. In this section, we will consider the cases when $K = 0$ and $K = -1$ assuming a subsolution.

\begin{lemma} \label{Lemma6-1}
Let $\psi(z)$ be a positive function defined on $\Omega$. For $z \in \Omega$ and a strictly locally convex function $v$ near $z$, if
\[\mathcal{G} [ v ] (z) =  F(a_{ij}[v])(z) = f (\kappa [v])(z) = \psi(z) \,\xi(v)(z) \]
then
\[ \mathcal{G}_v [ v ] (z) - \,\psi(z) \,\xi'(v)(z)  < 0\]
\end{lemma}

\begin{proof}
From \eqref{eq6-1} we have
\begin{equation*}
\begin{aligned}
\frac{\partial a_{ij}}{\partial v}\, =  & \, \frac{1}{w} \big(\, \eta'(v)\, \delta_{ij} \,+ \,\eta(v) \tilde{\gamma}^{ik} \nabla'_{kl} v  \,\tilde{\gamma}^{lj}\big) \\
\, = & \, \frac{\eta'^2(v) - \eta^2(v)}{w \eta'(v)} \,\delta_{ij} + \frac{\eta(v)}{\eta'(v)} a_{ij} \,=\,\frac{K}{w \eta'(v)}\, \delta_{ij} + \frac{\eta(v)}{\eta'(v)} a_{ij}
\end{aligned}
\end{equation*}
Therefore
\begin{equation*}
\mathcal{G}_v =  \,\frac{K}{w \eta'(v)}\, \sum f_i + \frac{\eta(v)}{\eta'(v)} \,F^{ij} a_{ij} \\
=  \,\frac{K}{w \eta'(v)}\, \sum f_i + \frac{\eta(v)}{\eta'(v)} \,\sum f_i \kappa_i
\end{equation*}
Since $ \sum f_i \kappa_i \leq \psi(z) \,\xi(v)$ by the concavity of $f$ and $f(0) = 0$,
\[ \mathcal{G}_v [ v ] - \,\psi(z) \,\xi'(v) \leq  \frac{K}{w \eta'(v)}\, \sum f_i + ( \frac{\eta(v)}{\eta'(v)} - \frac{\xi'(v)}{\xi(v)} ) \,\sum f_i \kappa_i  < 0 \]
\end{proof}

\begin{lemma}  \label{Lemma6-2}
For any fixed $t \in [0, 1]$, if $\underline{V}$ and $v$ are respectively strictly locally convex subsolution and solution to \eqref{eq6-8}, then $v \geq \underline{V}$.  Thus the Dirichlet problem \eqref{eq6-8} has at most one strictly locally convex solution.
\end{lemma}
\begin{proof}
If not, then $\underline{V} - v$ achieves a positive maximum at some $z_0 \in \Omega$. We have
\begin{equation} \label{eq6-14}
\underline{V}(z_0) > v(z_0),\quad \nabla'\underline{V}(z_0) = \nabla' v(z_0), \quad \nabla'^2\underline{V}(z_0) \leq \nabla'^2 v(z_0)
\end{equation}
Consider the deformation $v[s] := (1 - s)\, v + s \,\underline{V}$ for $s \in [0, 1]$.  In view of \eqref{eq6-1}, we can verify that $v[s]$ is strictly locally convex near $z_0$ for any $s \in [0, 1]$. In fact,  at $z_0$,
\begin{equation*}
\begin{aligned}
 & \,\, \eta(v[s]) \,\delta_{ij} \,+ \,\eta'(v[s]) \,\, \tilde{\gamma}^{ik} \,\,  \nabla'_{kl} v[s] \,\, \tilde{\gamma}^{lj}\,  \geq \,  \eta(v[s]) \,\delta_{ij} \,+ \,\eta'(v[s]) \, \tilde{\gamma}^{ik} \,\,\nabla'_{kl} \underline{V} \,\, \tilde{\gamma}^{lj} \\
 = \,\, & \eta'(v[s]) \,\Big( \frac{\eta(v[s])}{\eta'(v[s])} - \frac{\eta(\underline{V})}{\eta'(\underline{V})}\, \Big) \,\delta_{ij}
  +  \,\frac{\eta'(v[s])}{\eta'(\underline{V})} \Big( \eta(\underline{V}) \,\delta_{ij} \,+ \,\eta'(\underline{V})\,\tilde{\gamma}^{ik} \,\,\nabla'_{kl} \underline{V} \,\, \tilde{\gamma}^{lj} \Big ) > 0
\end{aligned}
\end{equation*}
where the last inequality is true since
\[\Big(\frac{\eta}{\eta'}\Big)' (v) \leq  0. \]

Now we define a differentiable function of $s \in [0, 1]$,
\[ a(s) := \mathcal{G} \Big[ v[s] \Big] \,-\, \Big(  ( 1 - t ) \frac{\mathcal{G}[ \underline{v} ]}{ \xi (\underline{v}) } + t \,\epsilon \Big)  \,\xi ( v[s]) \,\Big\vert_{z_0} \]
Since
\[ a(0) = \mathcal{G} [ v ] \,-\, \Big( ( 1 - t ) \frac{\mathcal{G}[ \underline{v} ]}{ \xi (\underline{v}) } + t \,\epsilon  \Big)  \,\xi ( v ) \, = 0 \]
and
\[ a(1) = \mathcal{G} [\underline{V}] \,-\, \Big(  ( 1 - t ) \frac{\mathcal{G}[ \underline{v} ] }{ \xi (\underline{v}) } + t \,\epsilon \Big)  \,\xi (\underline{V}) \, \geq 0,\]
there exists $s_0 \in [0, 1]$ such that $a(s_0) = 0$ and $a'(s_0) \geq 0$, that is,
\begin{equation}\label{eq6-15}
\mathcal{G} \Big[ v[s_0] \Big] (z_0) \,=\, \Big( ( 1 - t ) \frac{\mathcal{G}[ \underline{v} ]}{ \xi (\underline{v}) }  + t \,\epsilon \Big)  \,\xi ( v[s_0] ) (z_0) \,
\end{equation}
and
\begin{equation} \label{eq6-16}
\begin{aligned}
&\mathcal{G}^{ij}\Big[  v[s_0]  \Big] \nabla'_{ij}  (\underline{V} - v)(z_0)
 + \mathcal{G}^i \Big[   v[s_0]  \Big] \nabla'_i  (\underline{V} - v)(z_0)
\\ & +  \left(\mathcal{G}_v \Big[  v[s_0] \Big] - \Big( ( 1 - t ) \frac{\mathcal{G}[\underline{v}] }{ \xi (\underline{v}) } + t \,\epsilon  \Big) \xi'( v[s_0] ) \right)  (\underline{V} - v)(z_0) \geq 0
\end{aligned}
\end{equation}
However, by \eqref{eq6-14}, \eqref{eq6-15} and Lemma \ref{Lemma6-1}, the above expression should be strictly less than $0$, which is a contradiction.
\end{proof}

\vspace{2mm}

\begin{thm} \label{Theorem6-1}
For any $t \in [0, 1]$, the Dirichlet problem \eqref{eq6-8} has a unique strictly locally convex solution $v$, which satisfies $v \geq \underline{v}$ in $\Omega$.
\end{thm}

\begin{proof}
Uniqueness is proved in Lemma \ref{Lemma6-2}. We prove the existence using standard continuity method. Recall that $u$ and $v$ are related by transformation \eqref{eq6-2}. Hence the $C^2$ estimate \eqref{eq4-2} established in section 3 and 4 implies the $C^2$ bound for strictly locally convex solutions $v$ of \eqref{eq6-8} with $v \geq \underline{v}$, which in turn gives an upper bound for all principal curvatures of the radial graph.
Since $f = 0$ on $\partial\Gamma_n$, the principal curvatures admit a uniform positive lower bound,
which implies that equation \eqref{eq6-8} is uniformly elliptic for strictly locally convex solutions $v$ with $v \geq \underline{v}$.
We can then apply Evans-Krylov theory \cite{Evans, Krylov} to obtain
\begin{equation} \label{eq6-10}
\Vert v \Vert_{C^{2, \alpha} (\overline{\Omega})}  \leq C
\end{equation}
where $C$ is independent of $t$.

Now we consider
\[ C_0^{2, \alpha} (\overline{\Omega}) := \{ w \in C^{2, \alpha}( \overline{\Omega} ) \,| \,w = 0 \,\, \mbox{on} \,\, \partial\Omega \}, \]
which is a subspace of $C^{2, \alpha}( \overline{\Omega} )$. Obviously,
\[ \mathcal{U} := \left\{ w \in C_0^{2, \alpha} (\overline{\Omega}) \,\Big| \, \underline{v} + w \,\,\mbox{is}\,\,\mbox{stricly}\,\,\mbox{locally}\,\,\mbox{convex}\,\,  \right\} \]
is an open subset of $C_0^{2, \alpha} (\overline{\Omega})$.
Construct a map $\mathcal{L}: \,\mathcal{U} \times [ 0, 1 ] \rightarrow C^{\alpha}(\overline{\Omega})$,
\[ \mathcal{L} ( w, t ) = \mathcal{G} [ \underline{v} + w] \, - \, \Big( ( 1 - t ) \frac{\mathcal{G}[\underline{v}]}{ \xi (\underline{v}) }  + t \,\epsilon  \Big)  \,\xi (\underline{v} + w)  \]
Set
\[ \mathcal{S} = \{ t \in [0, 1] \,|\, \mathcal{L}(w, t) = 0 \,\,\mbox{has}\,\,\mbox{a}\,\,\mbox{solution}\,\,\mbox{in}\,\,\mathcal{U}\, \} \]

First note that
\[ \mathcal{L}(0, 0) = \mathcal{G} [ \underline{v} ] \, -  \frac{\mathcal{G}[\underline{v}]}{ \xi (\underline{v}) }  \,\xi (\underline{v}) = 0  \]
hence $0 \in \mathcal{S}$ and $\mathcal{S} \neq \emptyset$.

$\mathcal{S}$ is open in $[0, 1]$. In fact, for any $t_0 \in \mathcal{S}$, there exists $w_0 \in \mathcal{U}$ such that $\mathcal{L} ( w_0, t_0 ) = 0$. The Fr\'echet derivative of $\mathcal{L}$ with respect to $w$ at $(w_0, t_0)$ is a linear elliptic operator from $C^{2, \alpha}_0 (\overline{\Omega})$ to $C^{\alpha}(\overline{\Omega})$,
\[
\begin{aligned}
\mathcal{L}_w \big|_{(w_0, t_0)} ( h )  \,\,=  \,\,  \mathcal{G}^{ij} [  \underline{v} + w_0 ] \,\nabla'_{ij} h  + \mathcal{G}^i [ \underline{v} + w_0 ] \, \nabla'_i h   \\
+ \left(\mathcal{G}_v [  \underline{v} + w_0 ] - \Big( ( 1 - t_0 ) \frac{\mathcal{G}[\underline{v}]}{ \xi (\underline{v}) } + t_0\, \epsilon  \Big) \, \xi'(\underline{v} + w_0) \right) h
\end{aligned}
\]
By Lemma \ref{Lemma6-1}, $\mathcal{L}_w \big|_{(w_0, t_0)}$ is invertible. Hence by implicit function theorem, a neighborhood of $t_0$ is also contained in $\mathcal{S}$.

$\mathcal{S}$ is closed in $[0, 1]$. Let $t_i$ be a sequence in $\mathcal{S}$ converging to $t_0 \in [0, 1]$ and $w_i \in \mathcal{U}$ be the unique  solution associated with $t_i$ (the uniqueness is guaranteed by Lemma \ref{Lemma6-2}), i.e. $\mathcal{L} (w_i, t_i) = 0$. Since $\underline{v}$ is a subsolution of \eqref{eq6-8} in view of \eqref{eq6-20},  by Lemma \ref{Lemma6-2}, $w_i \geq 0$. Then by \eqref{eq6-10} we see that $v_i := \underline{v} + w_i$ is a bounded sequence in $C^{2, \alpha}(\overline{\Omega})$. Possibly passing to a subsequence $v_i$ converges to a strictly locally convex solution $v_0$ of \eqref{eq6-8} as $i \rightarrow \infty$. Obviously $w_0 := v_0 - \underline{v} \in \mathcal{U}$ and $\mathcal{L}(w_0, t_0) = 0$. Thus $t_0 \in \mathcal{S}$.
\end{proof}

Now we assume that $\underline{v}$ is not a solution of \eqref{eq2-26}, for otherwise we are done.

\begin{lemma} \label{Lemma6-3}
Let $v$ be a strictly locally convex solution of \eqref{eq6-9}. If $v \geq \underline{v}$ in $\Omega$, then
$v > \underline{v}$ in $\Omega$ and ${\bf n}(v - \underline{v}) > 0$ on $\partial\Omega$, where ${\bf n}$ is the interior unit normal to $\partial\Omega$.
\end{lemma}

The proof of Lemma \ref{Lemma6-3} is in the Appendix since it is long.

\vspace{2mm}

\begin{thm} \label{Theorem6-2}
For any $t \in [0, 1]$, the Dirichlet problem \eqref{eq6-9} has a strictly locally convex solution $v$ satisfying $v \geq \underline{v}$ in $\Omega$. In particular, \eqref{eq2-26} has a strictly locally convex solution $v$ satisfying $v \geq \underline{v}$ in $\Omega$.
\end{thm}

\begin{proof}
Since $f = 0$ on $\partial\Gamma_n$, the $C^{2, \alpha}$ estimate for strictly locally convex solutions $v$ of \eqref{eq6-9} with $v \geq \underline{v}$ can be established in view of \eqref{eq6-2} and \eqref{eq4-2}, which in turn yields $C^{4, \alpha}$ estimate by classical Schauder theory
\begin{equation} \label{eq6-12}
\Vert v \Vert_{C^{4,\alpha}(\overline{\Omega})} < C_4
\end{equation}
Besides, we have the estimate (see the expression in \eqref{eq6-1}),
\begin{equation} \label{eq6-13}
C_2^{-1} I <  \tilde{A}[v]  : = \{ \eta'(v) \nabla'_{ij} v + \eta(v) v_i v_j +  \eta(v) \,\delta_{ij}\}  < C_2 I    \quad \mbox{in} \quad \overline{\Omega}
\end{equation}
where $C_2$ and $C_4$ are independent of $t$.

Let $ C_0^{\,4,\alpha} (\overline{\Omega})$ be the subspace of $C^{4,\alpha}( \overline{\Omega} )$ defined by
\[ C_0^{ 4, \alpha} (\overline{\Omega}) := \{ w \in C^{ 4, \alpha}( \overline{\Omega} ) \,| \,w = 0 \,\, \mbox{on} \,\, \partial\Omega \} \]
and consider the bounded open subset
\[ \mathcal{O} := \left\{ w \in C_0^{4, \alpha} (\overline{\Omega}) \,\left\vert\,\begin{footnotesize}\begin{aligned} & w > 0 \,\,\mbox{in}\,\,\Omega, \quad\quad \nabla'_{\bf n} w > 0 \,\,\mbox{on}\,\, \partial\Omega,\\ & C_2^{-1} I < \tilde{A}[\underline{v} + w] < C_2 I   \,\, \mbox{in} \,\, \overline{\Omega} \\ & \Vert w {\Vert}_{C^{4,\alpha}(\overline{\Omega})} < C_4 + \Vert\underline{v}\Vert_{C^{4,\alpha}(\overline{\Omega})} \end{aligned}\end{footnotesize} \right.\right\} \]
Construct a map
$\mathcal{M}_t (w):  \,\mathcal{O} \times [ 0, 1 ] \rightarrow C^{2,\alpha}(\overline{\Omega})$
\[ \mathcal{M}_t (w) \,= \,\mathcal{G} [ \underline{v} + w ] \,- ( 1 - t ) \,\epsilon \,\,\xi (\underline{v} + w)\, -\,  t \,\, \psi(z, \underline{v} + w, \nabla' (\underline{v} + w))\]
Let $v^0$ be the unique solution of \eqref{eq6-8} at $t = 1$ (the existence and uniqueness are guaranteed by Theorem \ref{Theorem6-1}). Note that $v^0$ is also the solution of \eqref{eq6-9} when $t = 0$. Set $w^0 = v^0 - \underline{v}$. By Lemma \ref{Lemma6-2}, we have $w^0 \geq 0$ in $\Omega$, which in turn implies that $w^0 > 0$ in $\Omega$ and $\nabla'_{\bf n} w^0 > 0$ on $\partial\Omega$ by Lemma \ref{Lemma6-3}. Also note that $v^0$ satisfies \eqref{eq6-12} and \eqref{eq6-13}. Thus, $w^0 \in \mathcal{O}$. From
Lemma \ref{Lemma6-3}, \eqref{eq6-12} and \eqref{eq6-13} we observe that $\mathcal{M}_t(w) = 0$ has no solution on $\partial\mathcal{O}$ for any $t \in [0, 1]$. Besides, $\mathcal{M}_t$ is uniformly elliptic on $\mathcal{O}$ independent of $t$. Hence the degree of $\mathcal{M}_t$ on $\mathcal{O}$ at $0$
\[ \deg (\mathcal{M}_t, \mathcal{O}, 0)  \]
is well defined and independent of $t$. Therefore we only need to compute
$ \deg (\mathcal{M}_0, \mathcal{O}, 0) $.

Note that $\mathcal{M}_0 ( w ) = 0$ has a unique solution $w^0 \in \mathcal{O}$, and the Fr\'echet derivative of $\mathcal{M}_0$ with respect to $w$ at $w^0$ is a linear elliptic operator from $C^{4, \alpha}_0 (\overline{\Omega})$ to $C^{2, \alpha}(\overline{\Omega})$,
\[
\mathcal{M}_{0,w} |_{w^0} ( h )  =  \,  \mathcal{G}^{ij}[ v^0 ] \,\nabla'_{ij} h  + \mathcal{G}^i [ v^0 ]\,  \nabla'_i h   + \Big(\mathcal{G}_v [ v^0 ]  - \,\epsilon \,\,\xi' (v^0) \Big) \,h
\]
By Lemma \ref{Lemma6-1}
\[ \mathcal{G}_v [  v^0 ]  - \,\epsilon \,\,\xi' (v^0) < 0 \quad\mbox{in} \quad \Omega \]
Hence $\mathcal{M}_{0,w} |_{w^0}$ is invertible. By the degree theory in \cite{Li89},
\[ \deg (\mathcal{M}_0, \mathcal{O}, 0) = \deg( \mathcal{M}_{0, w}\vert_{w^0}, B_1, 0 ) = \pm 1 \neq 0 \]
where $B_1$ is the unit ball in $C_0^{4,\alpha}(\overline{\Omega})$. Thus
\[ \deg(\mathcal{M}_t, \mathcal{O}, 0) \neq 0 \quad\mbox{for}\,\,\mbox{all}\,\,t \in [0, 1]\]
which implies that the Dirichlet problem \eqref{eq6-9} has at least one strictly locally convex solution for any $t \in [0, 1]$. In particular, $t = 1$ solves the Dirichlet problem  \eqref{eq2-26}.
\end{proof}

\vspace{6mm}

\section{Existence in \, $\mathbb{S}^{n + 1}_{+}$}

\vspace{5mm}

For any $\epsilon > 0$,  we want to prove the existence of a strictly locally convex solution to the following Dirichlet problem when $K = 1$,
\begin{equation} \label{eq7-2}
\left\{
\begin{aligned}
G[u] := G( \nabla'^2 u, \nabla' u, u ) \,\,= & \,\, \psi(z, u, \nabla' u) - \epsilon  \quad\quad &\mbox{in} \quad \Omega \\
u \,\, = & \,\, \varphi   &\mbox{on} \quad \partial\Omega
\end{aligned} \right.
\end{equation}
Then a strictly locally convex solution to  \eqref{eq2-13} follows from the uniform ($\epsilon$-independent) $C^2$ estimates (established in Section 3 and 4) and approximation.

As we have seen from last section, there does not exist an auxiliary equation in $\mathbb{S}^{n+1}_+$ with an invertible linearized operator. Hence we want to build a continuity process starting from an auxiliary equation in $\mathbb{R}^{n+1}$.

For this, we first consider a continuous version of \eqref{eq3-18}. For $t \in [0, 1]$, denote
\[a^t_{ij} = \frac{- (\zeta^t)' \phi^t}{\sqrt{(\phi^t)^2 + (\zeta^t)'^2 |\nabla' u|^2}} (\gamma^t)^{ik}\, ( \nabla'_{kl} u + u \,\delta_{kl} )\,(\gamma^t)^{lj}\]
where
\[(\gamma^t)^{ik} =   \,\frac{1}{\phi^t}\, \Big( \delta_{ik} - \frac{ (\zeta^t)'^2(u) u_i u_k }{\sqrt{(\phi^t)^2 + (\zeta^t)'^2(u) |\nabla' u|^2 } ( \phi^t + \sqrt{(\phi^t)^2 + (\zeta^t)'^2(u) |\nabla' u|^2 } )}\Big)\]
and
\[ \phi^t(\rho) = \frac{\sin (t \rho)}{t} \quad \quad \quad\mbox{with} \quad\quad \rho \in \Big( 0, \,\frac{\pi}{2 \,t} \Big) \]
\[ \zeta^t(u) = \frac{1}{t} \, \mbox{arccot} \frac{u}{t} \quad\quad\quad\mbox{with} \quad\quad u \in (0, \infty) \]
Note that these geometric quantities on $\Sigma$ correspond to the background metric
\[ \bar{g}^t = d \rho^2 + (\phi^t)^2(\rho)\, \sigma \]
Geometrically, $(\mathbb{R}^{n+1}, \bar{g}^t)$ is the upper hemisphere $\mathbb{S}^{n+1}_+(\frac{1}{t})$ with center $0$ and radius $\frac{1}{t}$. The corresponding sectional curvature is $K^t = t^2$. As $t$ varies from $0$ to $1$, $\bar{g}^t$ provides a deformation from $\mathbb{R}^{n+1}$ to $\mathbb{S}^{n+1}_+$.

Define
\[ G^t [ u ] \, = \, G^t( \nabla'^2 u, \nabla' u, u ) \, = \,F( a^t_{ij} ) \]
Hence $G^1 = G$.
The following property is true by direct calculation.
\begin{prop}
$G^t [u]$ is increasing with respect to $t$.
\end{prop}
\begin{proof}
\[ a_{ij}^t = \,\Big(1 + \frac{|\nabla' u|^2}{u^2 + t^2} \Big)^{- \frac{1}{2}}\, \,\tilde{\gamma}^{ik}\,\big( \nabla'_{kl} u \,+\, u \,\delta_{kl} \big)\,\tilde{\gamma}^{lj} \]
where
\[ \tilde{\gamma}^{ik} \,=\, \delta_{ik} - \frac{u_i \,u_k}{\sqrt{u^2 + t^2 + |\nabla' u|^2} \,\big( \sqrt{u^2 + t^2} + \sqrt{u^2 + t^2 + |\nabla' u|^2} \big)} \]

\[ \begin{aligned}
\,\frac{\partial}{\partial t} \,G^t[ u ]  = \,& \Big( 1 + \frac{|\nabla' u|^2}{u^2 + t^2} \Big)^{- 3/2}F^{ij}  \cdot \\ &  \Big(  \frac{t\, |\nabla' u|^2}{(u^2 + t^2)^2}  \tilde{\gamma}^{ik} +  2 \Big( 1 + \frac{|\nabla' u|^2}{u^2 + t^2} \Big) \frac{\partial \tilde{\gamma}^{ik}}{\partial t} \Big) \big( \nabla'_{kl} u + u \delta_{kl} \big) \tilde{\gamma}^{lj}  \end{aligned} \]
The inverse $(\tilde{\gamma}_{ik})$ of $(\tilde{\gamma}^{ik})$ is given by
\[ \tilde{\gamma}_{ik} \,=\, \delta_{ik} + \frac{u_i \,u_k}{ u^2 + t^2 + \sqrt{ (u^2 + t^2)^2 + (u^2 + t^2) |\nabla' u|^2} }  \]
Since
\[  \frac{\partial \tilde{\gamma}^{ik}}{\partial t} \,=\, -  \tilde{\gamma}^{ip}\,\frac{\partial \tilde{\gamma}_{pq}}{\partial t} \, \tilde{\gamma}^{qk} \]
\[ \frac{\partial \tilde{\gamma}_{pq}}{\partial t} = \, - \frac{u_p\, u_q \,t}{(u^2 + t^2)^{3/2} \sqrt{u^2 + t^2 + |\nabla' u|^2}}\]
and
\[ \tilde{\gamma}^{ik} \,u_k \,=\, \frac{\sqrt{u^2 + t^2}}{\sqrt{u^2 + t^2 + |\nabla' u|^2}} \, u_i \]
therefore,
\[ \begin{aligned} \,\frac{\partial}{\partial t}\, G^t[ u ] = & \frac{t}{\sqrt{u^2 + t^2} \big( u^2 + t^2 + |\nabla' u|^2 \big)^{3/2} }    F^{ij} \cdot \\ & \Big( |\nabla' u|^2 \delta_{iq} + 2 u_i u_q \Big) \tilde{\gamma}^{qk}\, \big( \nabla'_{kl} u + u \delta_{kl} \big)\, \tilde{\gamma}^{lj} \geq 0
\end{aligned} \]
\end{proof}
Likewise, we define
\[
 \begin{aligned}
  \psi^t [ u ] := \,& \psi^t(z, u, \nabla' u) = \, \psi^t(V^t, \nu^t) \\ = \,& \psi^t\left( \frac{1}{\sqrt{u^2 + t^2}} \,z, \,\,\frac{\sqrt{u^2 + t^2}}{\sqrt{u^2 + t^2 + |\nabla' u|^2}} \Big( \nabla' u + z \Big) \right)
 \end{aligned} \]
Note that $z = \frac{\partial}{\partial\rho}$.

Recall that we have assumed a strictly locally convex subsolution.
\[\left\{
\begin{aligned}
G[\underline{u}] \,\,\geq & \,\,\psi(z, \underline{u}, \nabla' \underline{u})  \quad\quad &\mbox{in} \quad \Omega \\
\underline{u} \,\, = & \,\, \varphi   &\mbox{on} \quad \partial\Omega
\end{aligned} \right.
\]
Choose $\epsilon$ small such that
\[ \epsilon < \min \Big\{   \min\limits_{\overline{\Omega}} G^0[\underline{u}], \,\, \frac{1}{2}\,\min\limits_{[0, 1] \times \overline{\Omega}  \times [C_0^{-1}, C_0] \times \{ p \,\in \mathbb{R}^n \vert\, |p| \leq C_1 \}} \psi^t \Big\}  \]
By continuity, for $t \in [1 - \delta_1, \,1]$ where $\delta_1$ is a sufficiently small positive constant depending on $\epsilon$, we have
\begin{equation} \label{eq7-3}
\left\{
\begin{aligned}
G^t[\underline{u}] \,\,> & \,\, \psi^t [ \underline{u} ] - \,\frac{\epsilon}{2}  \quad\quad &\mbox{in} \quad \Omega \\
\underline{u} \,\, = & \,\, \varphi   &\mbox{on} \quad \partial\Omega
\end{aligned} \right.
\end{equation}
Denote $\mathcal{G}^t [\, v\, ] \,:= \mathcal{G}^t( \nabla'^2 v, \nabla' v, v )\, =: \, G^t [\, e^v \,]$. Consider the continuity process,
\begin{equation} \label{eq7-1}
\left\{ \begin{aligned}
\mathcal{G}^t [\, v\, ] \,  =  & \, \, \big( 1 - T(t) \big) \,\delta_2 \, e^{2 v} \, + \, T(t) \,\Big( \psi^t [\, e^v \,] - \epsilon \Big) \quad\quad & \mbox{in} \quad \Omega \\ v \,  = & \, \,\ln \varphi \quad \quad & \mbox{on} \quad \partial\Omega \end{aligned} \right.
\end{equation}
where $\delta_2$ is a small positive constant such that
\[ \delta_2  \,\max\limits_{\overline{\Omega}} \underline{u}^2 \,<\, \frac{\epsilon}{2} \]
and $T(t)$ is a smooth strictly increasing function with $T(0) = 0$,  $T(1) = 1$ satisfying
\[  \min\limits_{\overline{\Omega}} \,G^0 [\underline{u}] \, > \,2 \,\,T(1 - \delta_1)\, \,\max\limits_{ [0, 1] \times \overline{\Omega}}
\psi^t [ \underline{u}]  \]

\begin{prop}
$\underline{v} = \ln \underline{u}$ is a strict subsolution of \eqref{eq7-1} for any $t\in [0, 1]$.
\end{prop}
\begin{proof}
For $t \in [1 - \delta_1,  1]$,
\[ \begin{aligned}
\mathcal{G}^t [ \underline{v} ]  \,= &\, G^t [\underline{u}]  \, >   \, \psi^t [ \underline{u} ] - \frac{\epsilon}{2} \,> \, \delta_2  \, \underline{u}^2 +   \, \big( \psi^t [ \underline{u} ] - \epsilon \big) \\
\,\geq  & \, \big( 1 - T(t) \big) \, \delta_2  \, e^{2 \underline{v}} \, + \,  T(t)\,\, \big( \psi^t [ e^{\underline{v}} ] - \epsilon \big)
\end{aligned} \]
For $t \in [0, 1 - \delta_1]$,
\[ \begin{aligned}
\mathcal{G}^t [ \underline{v} ]  \,= & \, G^t [ \underline{u} ] \, \geq  \, G^0 [\underline{u}] \, > \, \frac{\epsilon}{2} \, +  \,T(1 - \delta_1) \, \psi^t [ \underline{u} ] \\ \geq &
 \,  \big( 1 - T(t) \big) \, \delta_2  \, \underline{u}^2 \, + \,  T(t)\,\, \big( \psi^t [ \underline{u} ] - \epsilon \big) \\ = &\, \big( 1 - T(t) \big) \, \delta_2  \, e^{2 \underline{v}} \, + \,  T(t)\,\, \big( \psi^t [ e^{\underline{v}} ] - \epsilon \big)
\end{aligned}
\]
\end{proof}

Now we can obtain the existence results in $\mathbb{S}^{n+1}_+$.

\begin{thm} \label{Theorem7-1}
For any $t \in [0, 1]$, the Dirichlet problem \eqref{eq7-1} has a strictly locally convex solution $v$ with $v \geq \underline{v}$ in $\Omega$. In particular, \eqref{eq7-2} has a strictly locally convex solution $u$ satisfying $u \geq \underline{u}$ in $\Omega$ when $K = 1$.
\end{thm}

\begin{proof}
The $C^{2, \alpha}$ estimates for strictly locally convex solutions $v$ of \eqref{eq7-1} with $v \geq \underline{v}$ is equivalent to the $C^{2, \alpha}$ estimates for strictly locally convex solutions $u$ with $u \geq \underline{u}$ to the Dirichlet problem
\begin{equation} \label{eq6-21}
\left\{ \begin{aligned}
G^t [\, u\, ] \,  =  & \, \, \big( 1 - T(t) \big) \,\delta_2 \, u^2 \, + \, T(t) \,\big( \psi^t [ u ] - \epsilon \big) \quad\quad & \mbox{in} \quad \Omega \\ u \,  = & \, \,\varphi \quad \quad & \mbox{on} \quad \partial\Omega \end{aligned} \right.
\end{equation}
which can be established by changing $\phi$ and $\zeta$ into $\phi^t$ and $\zeta^t$ in section 3, 4.  Then $C^{4, \alpha}$ estimate follows by classical Schauder theory. Thus we have the $t$-independent uniform estimates,
\begin{equation} \label{eq7-4}
\Vert v \Vert_{C^{4,\alpha}(\overline{\Omega})} < C_4 \quad \quad \mbox{and} \quad \quad
C_2^{-1}\, I <   \,\{ v_{ij} \,+\,v_i \, v_j\, + \,\delta_{ij} \} \,  < C_2\, I    \quad \mbox{in} \quad \overline{\Omega}
\end{equation}
Consider the subspace of $C^{4,\alpha}( \overline{\Omega} )$ given by
\[ C_0^{ 4, \alpha} (\overline{\Omega}) := \{ w \in C^{ 4, \alpha}( \overline{\Omega} ) \,| \,w = 0 \,\, \mbox{on} \,\, \partial\Omega \} \]
and the bounded open subset
\[ \mathcal{O} := \left\{ w \in C_0^{4, \alpha} (\overline{\Omega}) \,\left\vert\,\begin{footnotesize}\begin{aligned} & w > 0 \,\,\mbox{in}\,\,\Omega, \quad\quad \nabla'_{\bf n}\, w > 0 \,\,\mbox{on}\,\, \partial\Omega,\\ & C_2^{-1} \,I <  \, \{ (\underline{v} + w)_{ij} \,+\,(\underline{v} + w)_i \, (\underline{v} + w)_j\, + \,\delta_{ij} \} \,  < C_2 \, I   \,\, \mbox{in} \,\, \overline{\Omega} \\ & \Vert w {\Vert}_{C^{4,\alpha}(\overline{\Omega})} < C_4 + \Vert\underline{v}\Vert_{C^{4,\alpha}(\overline{\Omega})} \end{aligned}\end{footnotesize} \right.\right\} \]
Construct a map
$\mathcal{M}_t (w):  \,\mathcal{O} \times [ 0, 1 ] \rightarrow C^{2,\alpha}(\overline{\Omega})$,
\[ \mathcal{M}_t ( w ) = \mathcal{G}^t [ \underline{v} + w ] \, - \big( 1 - T(t) \big) \,\delta_2 \, e^{2 (\underline{v} + w )} - \, T(t) \, \Big( \psi^t [ e^{\underline{v} + w} ] - \epsilon \Big)\]
At $t = 0$, by Theorem \ref{Theorem6-1} for the case $K = 0$, there is a unique solution $v^0$ to
\eqref{eq7-1}.
By Lemma \ref{Lemma6-2} and Lemma \ref{Lemma6-3} we have $w^0 := v^0 - \underline{v}> 0$ in $\Omega$ and $\nabla'_{\bf n}\, w^0 > 0$ on $\partial\Omega$.  Moreover, $v^0$ satisfies \eqref{eq7-4} and thus $w^0 \in \mathcal{O}$.
Also, Lemma \ref{Lemma6-3} and  \eqref{eq7-4} implies that $\mathcal{M}_t( w ) = 0$ has no solution on $\partial\mathcal{O}$ for any $t \in [0, 1]$.
Besides, $\mathcal{M}_t$ is uniformly elliptic on $\mathcal{O}$ independent of $t$. Therefore,
$\deg (\mathcal{M}_t, \mathcal{O}, 0)$, the degree of $\mathcal{M}_t$ on $\mathcal{O}$ at $0$,
is well defined and independent of $t$.
Hence it suffices to compute $\deg (\mathcal{M}_0, \mathcal{O}, 0)$.

Note that $\mathcal{M}_0 ( w ) = 0$ has a unique solution $w^0 \in \mathcal{O}$. The Fr\'echet derivative of $\mathcal{M}_0$ with respect to $w$ at $w^0$ is a linear elliptic operator from $C^{4, \alpha}_0 (\overline{\Omega})$ to $C^{2, \alpha}(\overline{\Omega})$,
\[
\mathcal{M}_{0, w} |_{w^0} ( h )  =  \,  (\mathcal{G}^0)^{ij}[ v^0 ] \,\nabla'_{ij} h  + (\mathcal{G}^0)^i [ v^0 ] \, \nabla'_i h   + \big( (\mathcal{G}^0)_v [v^0]  - \,2 \,\delta_2 \, e^{2 v^0} \big)\, h
\]
By Lemma \ref{Lemma6-1}
\[ (\mathcal{G}^0)_v [v^0]  - \,2 \,\delta_2 \, e^{2 v^0}\, < 0 \quad\mbox{in} \quad \Omega \]
Thus $\mathcal{M}_{0,w} |_{w^0}$ is invertible. Applying the degree theory in \cite{Li89},
\[ \deg (\mathcal{M}_0, \mathcal{O}, 0) = \deg( \mathcal{M}_{0, w} |_{w^0}, B_1, 0) = \pm 1 \neq 0 \]
where $B_1$ is the unit ball in $C_0^{4,\alpha}(\overline{\Omega})$. Thus
\[ \deg(\mathcal{M}_t, \mathcal{O}, 0) \neq 0 \quad\mbox{for}\,\,\mbox{all}\,\,t \in [0, 1]\]
and this theorem is proved.
\end{proof}

\vspace{4mm}

\section{Appendix: Proof of Lemma \ref{Lemma6-3}}

\vspace{4mm}

\begin{proof}
Recall that we have assumed that $\underline{v}$ is not a solution of \eqref{eq2-26}. By \eqref{eq2-29} and \eqref{eq6-20} we know that $\underline{v}$ is a strict subsolution of \eqref{eq6-9} when $t \in [0, 1)$, while it is a subsolution but not a solution of \eqref{eq6-9} when $t = 1$. It is relatively easy to prove the conclusion when $t \in [0, 1)$, following the ideas in \cite{Su16}. For the case $t = 1$:
\begin{equation*}
\left\{ \begin{aligned} \mathcal{G} [ v ] \,\, =  & \, \,  \psi(z, v, \nabla' v) \quad & \mbox{in} \quad \Omega \\
 v \,\,  = & \,\, \underline{v} \quad \quad & \mbox{on} \quad \partial\Omega \end{aligned} \right.
\end{equation*}
we will make use of the maximum principle which was originally discovered in \cite{Serrin}, while more precisely stated for our purposes in \cite{GNN} (see section 1.3, p. 212). Because the maximum principle and Hopf lemma there are designed for domains in Euclidean spaces, we need to rewrite the above equation in a local coordinate system of $\mathbb{S}^n$. For keeping the strict local convexity of the variations in our proof, we first transform the above equation back under the transformation \eqref{eq6-2} into a form as \eqref{eq2-13}:
\begin{equation} \label{eq5-7}
\left\{ \begin{aligned}
G(\nabla'^2 u, \nabla' u, u) \,\,= & \,\,\psi(z, u, \nabla' u)  \quad \quad  &\mbox{in} \quad & \Omega \\
u \,\, = & \,\,\underline{u}  \quad \quad & \mbox{on} \quad & \partial\Omega \end{aligned} \right.
\end{equation}
Recall that $G(\nabla'^2 u, \nabla' u, u) = F(A[u])$ where $A[u] = \{ \gamma^{ik} h_{kl} \gamma^{lj} \}$. Since at this time we do not use local orthonormal frame on $\mathbb{S}^n$, but rather a local coordinate system of $\mathbb{S}^n$, $\gamma^{ik}$ and $h_{kl}$ will appear different from \eqref{eq3-10} and \eqref{eq3-13}.

Meanwhile, the subsolution assumption \eqref{eq2-29} (i.e. \eqref{eq1-4}) can be rewritten as
\begin{equation*}
\left\{ \begin{aligned}
G(\nabla'^2 \underline{u}, \nabla' \underline{u}, \,\underline{u}) \geq & \,\,\psi(z, \underline{u}, \nabla' \underline{u})  \quad  &\mbox{in} \quad \Omega\\
\underline{u} = & \,\,\varphi  \quad \quad &\mbox{on} \quad \partial \Omega \end{aligned} \right.
\end{equation*}
Note that $\underline{u}$ is not a solution of \eqref{eq5-7}.

(i) We first show that if a strictly locally convex solution $u$ of \eqref{eq5-7} satisfies $u \geq \underline{u}$ in $\Omega$, then $u > \underline{u}$ in $\Omega$. Let $N \notin \Omega$ be the north pole of $\mathbb{S}^n$. Take the radial projection of $\mathbb{S}^n \setminus \{N\}$ onto $\mathbb{R}^n \times \{ -1 \} \subset \mathbb{R}^{n+1}$ and let $\tilde{\Omega}$ be the image of $\Omega$. We thus have a coordinate system $(x_1, \ldots, x_n)$ on $\mathbb{R}^n \times \{ -1 \} \cong \mathbb{R}^n$.  The metric on $\mathbb{S}^n$, its inverse, and the Christoffel symbols are given respectively by
\[ \sigma_{ij} = \frac{16}{\mu^2} \,\delta_{ij}, \quad \mu = 4 + \sum x_i^2, \quad \sigma^{ij} = \frac{\mu^2}{16}\, \delta_{ij} \]
\[ {\Gamma}_{ij}^k = - \frac{2}{\mu} \,(\delta_{ik} x_j + \delta_{jk} x_i - \delta_{ij} x_k) \]
Consequently, the metric on $\Sigma$, its inverse and the second fundamental form on $\Sigma$ are given respectively by (c.f. \cite{SX15})
\begin{equation*}
g_{ij} = \phi^2 \,\sigma_{ij} + \zeta'^2(u)\, u_i u_j
\end{equation*}
\begin{equation*}
g^{ij} = \frac{1}{\phi^2} \Big( \sigma^{ij} - \frac{\zeta'^2(u) u^i u^j}{\phi^2 + \zeta'^2(u) |\nabla' u|^2} \Big), \quad u^i = \sigma^{ik} u_k
\end{equation*}
\begin{equation*}
h_{ij} =  \frac{- \zeta'(u) \phi}{\sqrt{\phi^2 + \zeta'^2(u) |\nabla' u|^2}} ( \nabla'_{ij} u + u \,\sigma_{ij} )
\end{equation*}
The entries of the symmetric matrices $\{\gamma_{ik}\}$ and $\{\gamma^{ik}\}$ depend only on $x_1, \ldots, x_n$, $u$ and the first derivatives of $u$.

Now, set $\tilde{u} = \mu u$. By straightforward calculation,
\begin{equation} \label{eq5-10}
\nabla'_{ij} u + u \,\sigma_{ij} = \frac{1}{\mu} \tilde{u}_{ij} + \frac{2 \delta_{ij}}{\mu^2} \Big( \tilde{u} - \sum\limits_k x_k \tilde{u}_k \Big)
\end{equation}
As a result, \eqref{eq5-7} can be transformed into the following form:
\begin{equation*}
\left\{ \begin{aligned}
\tilde{G}(D^2\tilde{u}, D \tilde{u}, \tilde{u}, x_1, \ldots, x_n)  = F\Big( A\Big[ \frac{\tilde{u}}{\mu} \Big]\Big)\, = & \,\tilde{\psi}(x_1, \ldots, x_n, \tilde{u}, D \tilde{u})  \quad   &\mbox{in} \quad & \tilde{\Omega} \\
\tilde{u} \,\, = & \, \mu \,\underline{u}  \quad  & \mbox{on} \quad & \partial\tilde{\Omega} \end{aligned} \right.
\end{equation*}
where $\tilde{u}_i = \frac{\partial\tilde{u}}{\partial x_i}$, $D \tilde{u} = ( \tilde{u}_1, \ldots, \tilde{u}_n )$, $\tilde{u}_{ij} = \frac{\partial^2 \tilde{u}}{\partial x_i \partial x_j}$ and $D^2 \tilde{u} = \{ \tilde{u}_{ij} \}$.

Meanwhile, $\tilde{\underline{u}} := \mu \,\underline{u}$ satisfies
\begin{equation*}
\left\{ \begin{aligned}
\tilde{G}(D^2 \tilde{\underline{u}}, D \tilde{\underline{u}}, \tilde{\underline{u}}, x_1, \ldots, x_n) \,\,\geq & \,\,\tilde{\psi}(x_1, \ldots, x_n, \tilde{\underline{u}}, D \tilde{\underline{u}})  \quad  &\mbox{in} \quad \tilde{\Omega}\\
\tilde{\underline{u}}\,\, = & \,\,\mu \, \underline{u}  \quad \quad & \mbox{on} \quad \partial \tilde{\Omega} \end{aligned} \right.
\end{equation*}
Subtract the above two,
\[ \mathcal{L} \big( \tilde{\underline{u}} - \tilde{u} \big) : =  a_{ij} (x) \,(\tilde{\underline{u}} - \tilde{u})_{ij} + b_i(x) \, (\tilde{\underline{u}} - \tilde{u})_{i} + c(x) \,(\tilde{\underline{u}} - \tilde{u}) \geq 0 \]
where $ x : = ( x_1, \ldots, x_n )$,
\[ a_{ij} (x)  = \int_0^1 \tilde{G}^{ij}\, d s, \quad\quad b_i(x) = \int_0^1 \big( \tilde{G}^i - \tilde{\psi}^i \big) \,d s, \quad\quad c(x) = \int_0^1 \big( \tilde{G}_{\tilde{u}} - \tilde{\psi}_{\tilde{u}} \big) \, d s,  \,\]
\[ \tilde{G}^{ij} : = \frac{\partial \tilde{G}}{\partial \tilde{u}_{ij}} \Big(D^2\tilde{u} + s\,D^2 (\tilde{\underline{u}} -\tilde{u}),\, D \tilde{u} + s\, D(\tilde{\underline{u}} - \tilde{u}), \,\tilde{u} + s\, (\tilde{\underline{u}} - \tilde{u}),\, x_1, \ldots, x_n \Big),   \]
\[ \tilde{\psi}^i : = \frac{\partial \tilde{\psi}}{ \partial \tilde{u}_i }\Big(x_1, \ldots, x_n, \,\tilde{u} + s \,(\tilde{\underline{u}} -\tilde{u}), \, D \tilde{u} + s \, D (\tilde{\underline{u}} -\tilde{u}) \Big), \]
\[ \mbox{and} \quad \tilde{G}^i = \frac{\partial\tilde{G}}{\partial \tilde{u}_i}, \quad  \tilde{G}_{\tilde{u}} = \frac{\partial\tilde{G}}{\partial \tilde{u}}, \quad\tilde{\psi}_{\tilde{u}} = \frac{\partial\tilde{\psi}}{\partial \tilde{u}} \quad  \mbox{can} \,\,\mbox{be} \,\,\mbox{defined}\,\, \mbox{similarly}.\]
In view of \eqref{eq5-10} and \eqref{eq2-14} we know that
\[ \tilde{G}^{ij} = \frac{\partial \tilde{G}}{\partial \tilde{u}_{ij}} = \frac{\partial F}{\partial a_{kl}} \,\frac{\partial a_{kl}}{\partial \tilde{u}_{ij}} = \frac{1}{\mu} \, \frac{\partial G}{\partial u_{ij}}  \]
Hence the linearized operator $\mathcal{L}$ is uniformly elliptic. Besides, its coefficients are uniformly bounded, which can be seen from the algebraic fact
\[ \Big( \frac{\partial \gamma^{11}}{\partial A}, \ldots, \frac{\partial \gamma^{1n}}{\partial A}, \ldots, \frac{\partial \gamma^{nn}}{\partial A} \Big)\, = \, \Big( \frac{\partial g^{11}}{\partial A}, \ldots, \frac{\partial g^{1n}}{\partial A}, \ldots, \frac{\partial g^{nn}}{\partial A} \Big) \big( I \otimes g^{-\frac{1}{2}} + g^{-\frac{1}{2}} \otimes I \big)^{-1}\]
where $g^{-\frac{1}{2}} = \{ \gamma^{ik} \}$ and $A$ can be $\tilde{u}$ or $\tilde{u}_i$ with $i = 1, \ldots, n$.

Therefore, we can apply the Maximum Principle (see p. 212 of \cite{GNN}) to conclude that $\tilde{u} > \tilde{\underline{u}}$ in $\tilde{\Omega}$, which immediately yields $u > \underline{u}$ in $\Omega$.

(ii) To prove ${\bf n}(u - \underline{u}) > 0$ on $\partial\Omega$, we pick an arbitrary point $z_0 \in \partial\Omega$ and assume $z_0$ to be the north pole of $\mathbb{S}^n \subset \mathbb{R}^{n + 1}$. We introduce a local coordinate system about $z_0$ by taking the radial projection of the upper hemisphere onto the tangent hyperplane of $\mathbb{S}^n$ at $z_0$ and identifying this hyperplane to $\mathbb{R}^n$. Denote the coordinates by $(y_1, \ldots, y_n)$  and assume that the positive $y_n$-axis is the interior normal direction to $\partial\Omega \subset \mathbb{S}^n$ at $z_0$. In this coordinate system, the metric on $\mathbb{S}^n$, its inverse, and the Christoffel symbols are given respectively by (see \cite{Oliker, GS93})
\[ \sigma_{ij} = \frac{1}{\mu^2} \Big( \delta_{ij} - \frac{y_i y_j}{\mu^2} \Big), \quad \mu = \sqrt{1 + \sum y_i^2}  \]
\[ \sigma^{ij} = \mu^2 ( \delta_{ij} + y_i y_j ) \]
\[ {\Gamma}_{ij}^k = - \frac{\delta_{ik} y_j + \delta_{jk} y_i}{\mu^2} \]
The metric $g_{ij}$, its inverse $g^{ij}$ and the second fundamental form $h_{ij}$ on $\Sigma$ have the form as above.
The entries of the symmetric matrices $\{\gamma_{ik}\}$ and $\{\gamma^{ik}\}$  depend only on $y_1, \ldots, y_n$, $u$ and the first derivatives of $u$.

Now set $\tilde{u} = \mu u$. By straightforward calculation we have
\begin{equation} \label{eq5-8}
\nabla'_{ij} u + u \,\sigma_{ij} = \mu^{- 1} \tilde{u}_{ij}
\end{equation}
Equation \eqref{eq5-7} can be transformed into an equation defined in an open neighborhood of $0$ on $\mathbb{R}^n$, which is the radial projection of a neighborhood of $z_0$ on $\mathbb{S}^n$:
\begin{equation*}
\tilde{G}(D^2\tilde{u}, D \tilde{u}, \tilde{u}, y_1, \ldots, y_n)  =  F\Big( A\Big[ \frac{\tilde{u}}{\mu} \Big]\Big) =  \,\tilde{\psi}(y_1, \ldots, y_n, \tilde{u}, D \tilde{u})
\end{equation*}
where $\tilde{u}_i = \frac{\partial\tilde{u}}{\partial y_i}$, $D \tilde{u} = ( \tilde{u}_1, \ldots, \tilde{u}_n )$, $\tilde{u}_{ij} = \frac{\partial^2 \tilde{u}}{\partial y_i \partial y_j}$ and $D^2 \tilde{u} = \{ \tilde{u}_{ij} \}$.

In view of \eqref{eq5-8} and \eqref{eq2-14} we know that
\[ \frac{\partial \tilde{G}}{\partial \tilde{u}_{ij}} = \frac{\partial F}{\partial a_{kl}} \,\frac{\partial a_{kl}}{\partial \tilde{u}_{ij}} = \frac{1}{\mu} \, \frac{\partial G}{\partial u_{ij}}  \]
By (i) and applying Lemma H (see p. 212 of \cite{GNN}) we find that $(\tilde{u} - \tilde{\underline{u}})_n (0) > 0$ and equivalently ${\bf n}(u - \underline{u}) (z_0) > 0$.

\end{proof}

\end{document}